%% file: koie2019.tex
\documentclass[reqno,draft]{amsart}
\usepackage{bm,amscd,textcomp,mathrsfs,lscape}
\usepackage{amsmath,amsthm,amssymb,verbatim,mathtools}
\usepackage{color}
\usepackage{enumerate}
\usepackage[dvipdfm]{graphicx}
\usepackage{latexsym}
\usepackage[all]{xy}
\usepackage{calc}%時間計算
\DeclareMathOperator{\Hom}{Hom}

\newtheorem{theorem}{Theorem}[section]
\newtheorem{lemma}[theorem]{Lemma}
\newtheorem{proposition}[theorem]{Proposition}
\newtheorem{corollary}[theorem]{Corollary}
\theoremstyle{remark}
\newtheorem{remark}{Remark}[section]
\theoremstyle{definition}
\newtheorem{definition}{Definition}
\newtheorem{example}{Example}
\numberwithin{equation}{section}

\title[Brenner's theorem for Hochschild extension algebras]
	{An application of a theorem of Sheila Brenner for Hochschild extension algebras of a truncated quiver algebra}
\author[H. Koie]{Hideyuki Koie}
\address[H. Koie]{Department of Mathematics,
					Tokyo University of Science, 1-3 Kagurazaka,
                    	Shinjuku-ku, Tokyo 162--8601, Japan}
\email{1114702@ed.tus.ac.jp}


\begin{document}
\maketitle
\begin{abstract}
Let $A$ be a truncated quiver algebra over an algebraically closed field
such that any oriented cycle in the ordinary quiver of $A$ is zero in $A$.
We give the number of the indecomposable direct summands of
the middle term of an almost split sequence for a class of Hochschild extension
algebras of $A$ by the standard duality module $D(A)$.
\end{abstract}
%%%%%%%%%%%%%%%%%%%%%%%%%%%%%%%%%%%%%%%%%%%%%%%%%%%%
%
%        本文
%
%%%%%%%%%%%%%%%%%%%%%%%%%%%%%%%%%%%%%%%%%%%%%%%%%%%%
\section{Introduction}
Brenner \cite{Brenner} studied the number of indecomposable direct summands of
the middle term of an almost split sequence and
she showed how to determine this number for an artin algebra. 
However, it is not easy to compute this number by original method.
Fern$\acute{{\rm a}}$ndez-Platzeck \cite{FP} obtained further information for
trivial extension algebras.
Similarly,
we will give an application of a theorem of Sheila Brenner for Hochschild extension algebras
which is a generalization of trivial extension algebras.

Let $K$ be an algebraically closed field and $A=K\Delta_A/I$ 
a bound quiver algebra, where $\Delta_A$ is a finite connected quiver
and the ideal $I$ is admissible.
We denote by $D(A)$ the standard duality module $\Hom_K(A,\,K)$.
We recall the definitions of a Hochschild extension
and a Hochschild extension algebra from \cite{Hochschild}, \cite{Koie} and \cite{Handbook of algebra}.
By a Hochschild extension over $A$ by $D(A)$, 
we mean an exact sequence
$$
	0 \longrightarrow
    	D(A) \stackrel{\kappa}{\longrightarrow}
        	T \stackrel{\rho}{\longrightarrow}
            	A \longrightarrow
                	0
$$
such that $T$ is a $K$-algebra, $\rho$ is an algebra epimorphism
and $\kappa$ is a $T$-bimodule monomorphism.
The algebra $T$ is called a Hochschild extension algebra.
It is well known that $T$ is isomorphic to $A \oplus D(A)$ with the multiplication
$$
  (a,\,f)(b,\,g) = (ab,\, ag + fb + \alpha(a,\,b)),
$$
where $\alpha: A \times A \longrightarrow D(A)$ is a $2$-cocycle.
We denote by $T_\alpha(A)$ the Hochschild extension algebra corresponding to
a $2$-cocycle $\alpha$.
Then, $T_0(A)$ is just the trivial extension algebra
$A \ltimes D(A)$.

For a hereditary algebra $A$,
Yamagata \cite{yamagata 1981} studied the Auslander-Reiten quivers of Hochschild extension algebras.
In particular,
for a Hochschild extension algebra $T$,
he showed that
the number of isomorphism classes of indecomposable modules in ${\rm mod}\,T$
is twice the one for ${\rm mod}\,A$.
However this does not hold for a Hochschild extension algebra for a general algebra.
So we are interested in an almost split sequence in the module category of Hochschild extension algebra.
 
In \cite{Brenner}, Brenner showed
how to determine the number of indecomposable
direct summands of the middle term of an almost split sequence
starting with a simple module. As a consequence of this result,
for a self-injective artin algebra, she obtained the number of indecomposable direct
summands of ${\rm rad}\,P/{\rm soc}\,P$, where $P$ is an indecomposable
projective module.
These results by Brenner
play an important role in the representation theory of algebras.
However,   
in general, it is not easy to compute these numbers for a given algebra.
So there is few works to compute these numbers.
In \cite{FP}, Fern$\acute{{\rm a}}$ndez and Platzeck gave a
simple interpretation of them in the particular case of the trivial extension
$T_0(A)$.
This is done by focusing on the number of nonzero cycles in $\Delta_{T_0(A)}$.
Fern$\acute{{\rm a}}$ndez and Platzeck proved that the set of nonzero cycles coincides with
the set of elementary cycles.
Using this fact, they gave the numbers considered by Brenner by computing
the cardinality of the equivalent classes of the set of nonzero cycles.

In this paper, for a truncated quiver algebra $A$ such that any
oriented cycle is zero in $A$,
we give a similar interpretation of the numbers considered by Brenner
for a Hochschild extension
algebra $T_\alpha(A)$ such that $\Delta_{T_\alpha(A)} = \Delta_{T_0(A)}$ holds.
Unfortunately, for a Hochschild extension algebra,
the set of nonzero cycles does not coincide with
the set of elementary cycles in general.
So by defining an {\it $\alpha$-revived cycle},
%we will define a {\it $\alpha$-revived cycle} and
we will prove that
a nonzero cycle in $T_\alpha(A)$ is
either an elementary cycle or an $\alpha$-revived cycle.
%In addition,
%we will give the numbers considered by Brenner 
%in a similar way to \cite{FP}.
%we can describe the ordinary quiver $\Delta_{T_\alpha(A)}$ of
%a Hochschild extension algebra for a truncated quiver algebra in a similar way to \cite{Koie}. 
So we enumerate these nonzero cycles and then
we can give the numbers considered by Brenner easily.

This paper is organized as follows:
In Section $2$, we recall results of Brenner, some definitions
and facts about Hochschild extension algebras.
In particular, some facts about $2$-cocycles and the ordinary
quiver of Hochschild extension algebras are based on
\cite{Koie} and \cite{Koie2}.
In Section $3$, we will give a characterization of nonzero oriented cycles in
a Hochschild extension algebra.
We will show that a nonzero oriented cycle in $T_\alpha(A)$ is either
an elementary cycle or an $\alpha$-revived cycle.
In  Section $4$, we give the number of indecomposable direct summands of
the middle term of almost split sequence for $T_\alpha(A)$.
We define a relation in the set of all nonzero cycles with an origin $h$ in $T_\alpha(A)$.
And we give the number by computing the cardinality of the set of
equivalence classes. 
Moreover,
for a class of Hochschild extension algebra,
we will show that
if the origin $h$ is neither sink nor source then
the number is one. 

For general facts on quivers and bound quiver algebras, we refer to \cite{ASS} and \cite{Fro}.
Also for Hochschild extension algebra,
we refer to \cite{Koie}, \cite{Koie2}, \cite{yamagata 1981} and \cite{Handbook of algebra}.
Moreover, the notation including $\Delta_0$, $\Delta_1$, $\Delta_+$ and the
isomorphism $\Theta : \bigoplus_{q} D(HH_{2,\,q}(A)) \xrightarrow{\sim} {H^2(A,\,D(A))}$
is same as in \cite{Koie}.
%%%%%%%%%%%%%%%%%%%%%%%%%%%%%%%%%%%%%%%%%%%%%%%%
%
%
%
%
%
%%%%%%%%%%%%%%%%%%%%%%%%%%%%%%%%%%%%%%%%%%%%%%%%%%%
\section{Preliminaries}
In this section, we recall theorems of Brenner \cite{Brenner},
the ordinary quiver of a Hochschild extension algebra along \cite{Koie}
and the definition of an elementary cycle in the ordinary quiver
of a Hochschild extension algebra introduced in \cite{FP}.
After that we define an $\alpha$-revived cycle
for a $2$-cocycle $ \alpha$ and we give an example of these cycles.

\subsection{A theorem of Brenner}

We recall Brenner's results along her paper \cite{Brenner}. 
Let $A$ be an artin algebra.
An element of $A$ of the form $a = xay$
will be called an arrow,
where $x$ and $y$ are primitive idempotents of
$A$ and $a \in {\rm rad}\,A \backslash {\rm rad}^2A$.
Let $e$ be a primitive idempotent of $A$. A set $\Lambda$ of arrows will be called a
complete set of arrows for ${\rm rad}\,Ae$ if
\begin{itemize}
 \item it generates ${\rm rad}\,Ae$ as a left $A$-module,
 \item no proper subset of $\Lambda$ generates ${\rm rad}\,Ae$.
\end{itemize}
A complete set  of arrows for ${\rm rad}\,eA$ is defined similarly.

Let $e$ be a primitive idempotent. We denote by $\mathscr{N}$ the set of
pair $(N,\,n)$ of integers such that there exist sets of arrows
$\Lambda_i$ and $\Gamma_i$, $0 \leq i \leq n$, of which only
$\Lambda_0$ and $\Gamma_0$ can be empty, satisfying the following conditions:
\begin{enumerate}
	\item $i \neq j$ 
			implies 
			$\Lambda_i \cap \Lambda_j = \emptyset = \Gamma_i \cap \Gamma_j$,
	\item $\bigcup_{i=0}^n \Lambda_i$ is a complete set of arrows for ${\rm rad}\,Ae$,
	\item $\bigcup_{i=0}^n \Gamma_i$ is a complete set of arrows for ${\rm rad}\,eA$,
	\item If $i \neq j$, or $i=0$, or $j=0$, then
			$a \in \Lambda_i$ and $b \in \Gamma_j$ implies $ab = 0$,
	\item $N = n + {\rm card}\,\Lambda_0$.
\end{enumerate}
Let $N_e = {\rm max}\{N \mid \text{there exists $n$ such that $(N, n) \in \mathscr{N}$} \}$
and
$n_e = {\rm min} \{n \mid (N_e, n) $ $\in \mathscr{N} \}$.

\begin{theorem}[\cite{Brenner}]\label{Brenner1}
	Let $S$ be a noninjective simple $A$-module, and
	let $e$ be a primitive idempotent of $A$ such that $S \cong eA/{\rm rad}\,eA$.
	The middle term of the almost split sequence starting at $S$ has exactly $N_e$
	indecomposable direct summands.
	Furthermore, the number of indecomposable projective direct summands
	is equal to $N_e- n_e$.
\end{theorem}

\begin{corollary}[\cite{Brenner}]
	Let $e$ be same as in Theorem \ref{Brenner1} and $P= eA$.
	If $A$ is self-injective, then the number of indecomposable direct summands
	of  ${\rm rad}\,P / {\rm soc}\,P$ is equal to $n_e$.
\end{corollary}

\subsection{A $2$-cocycle induced by a cycle in the ordinary quiver}\label{2-cocycle}
From now on,
let $K$ be an algebraically closed field, $\Delta$ a quiver and $A := K\Delta/R_\Delta^n\, (n\geq2)$ a
truncated quiver algebra such that any oriented cycle in $\Delta$ is zero in $A$. 
%If $A= K$, then a Hochschild extension algebra of $A$ is only trivial extension algebra.
We assume that ${\rm dim}\,A > 1$.

Since $A$ is a truncated quiver algebra,
we can take a set 
$\mathbb{M} := \{p_i \mid i= 1,\ldots,t \}$ of paths in $\Delta$ such that
$\{ \overline{p_i} \mid i= 1,\ldots,t \}$ is a basis of ${\rm soc}_{A^e}\,A$. 
Moreover,
let $\{ \overline{p_1}, \ldots, \overline{p_t}, \ldots, \overline{p_d} \}$ be a basis of $A$
by taking paths $p_{t+1}, \ldots, p_d$ in $\Delta$.
%We can obtain this basis as the set of paths of length below $n-1$. 
We denote by $\{ \overline{p_1}^*, \ldots, \overline{p_t}^*, \ldots, \overline{p_d}^* \}$
the dual basis in $D(A)$.
We note that,
by \cite[Proposition 2.2.]{FP},
the ordinary quiver 
$\Delta_{T_0(A)}$ is given by
\begin{itemize}
 \item $(\Delta_{T_0(A)})_0 = \Delta_0$,
 \item $(\Delta_{T_0(A)})_1 = \Delta_1 \cup \{y_{p_1} ,\ldots, y_{p_t} \}$,
\end{itemize}
where, for each $i$, $y_{p_i}$ is an arrow from $t(p_i)$ to $s(p_i)$.

Next,  
under the notation of \cite{Koie} and \cite{Koie2},
we will define a $2$-cocycle $\alpha$.
For $n+1 \leq s \leq 2n-2$,
%we have $HH_{2,\,s}(A) = K^{a_s}$,
%where $a_s = {\rm card}(\Delta_{s}^c/C_s)$.
let $\gamma = x_1 x_2 \cdots x_s \in \Delta_{s}^c$ be a cycle. 
Then it is easy to check that $\gamma$ is a basic cycle.
%Then we have $\gamma = \gamma_b^l$ for some basic cycle $\gamma_b$.
%We denote by $s$ the length of $\gamma_b$.
%If $l \geq 2$, by $n \leq s$, we have $2n \leq ln \leq ls = q$.
%This contradicts $q \leq 2n-2$.
%So we only consider that $l=1$, that is, $\gamma$ is a basic cycle and $q=s$.
We regard the subscripts $i$ of $x_i$ modulo $s\, (1 \leq i \leq s)$.
Moreover,
$((A \otimes_{A^e} \textrm{\boldmath{$P$}}_*)_s ,\, (\tilde{d}_*)_s)$ is
$\Delta_s^{c}/C_s$-graded
and
$
	\{ v_i =  x_{i+n} \cdots x_{i+s-1} \otimes_{K\Delta_0^e} x_i x_{i+1} \cdots x_{i+n-1}
		 \mid
			1 \leq i \leq s \}
$
is a basis of $((A \otimes_{K\Delta_0^e}K\Delta_n)_s)_{\overline{\gamma}}$.
We denote by $\{ v_i^* \mid 1 \leq i \leq s \}$ the dual basis in 
$D(((A \otimes _{K\Delta_0^e} K\Delta_n)_s)_{\overline{\gamma}})$.
Then
we have the following complex
\begin{align*}
D(((A \otimes_{K\Delta_0^e}K\Delta_1)_s)_{\overline{\gamma}})  
  \stackrel{0}\longrightarrow
	&D(((A \otimes_{K\Delta_0^e}K\Delta_n)_s)_{\overline{\gamma}}) \\
	 & \xrightarrow{D(((\tilde{d_3})_{s})_{\overline{\gamma}})}
		D(((A \otimes_{K\Delta_0^e}K\Delta_{n+1})_s)_{\overline{\gamma}}),
\end{align*}
and we have the following isomorphism
\begin{align*}
	D(HH_{2,\,s,\,{\overline{\gamma}} }(A))
    	\cong  %{\rm Ker}\,(D((\tilde{d}_{3})_q))/{\rm Im}\,(D((\tilde{d}_2)_q)) \notag \\
        	 {\rm Ker}\,(D(((\tilde{d_3})_{s})_{\overline{\gamma}}))
            = \langle v_1^* + \cdots + v_s^* \rangle. \label{case1}
\end{align*}
We denote the map
$\Theta(v_i^*) : A \times A \longrightarrow D(A)$
by $\alpha_i$
for $i = 1, 2, \ldots, s$.
Then each $\alpha_i$ is the map as follows:
\begin{align*}
	\alpha_i(\overline{a} ,\, \overline{b}) 
	%b_1 \cdots b_{m_1} &\otimes_K b_{m_1 +1} \cdots b_{m_1 + m_2} \notag \\
     	& =
            	\begin{cases}
                	\overline{x_{i+m} \cdots x_{i+s-1}}^* 
                    	& \text{if $\overline{a},\, \overline{b} \neq 0$ in $A$,
                    	  $\: n \leq m < s$} \\
                    	&\quad \quad\quad \quad\text{and $ab = x_i\cdots x_{i+m-1}$,} \\
                    \overline{s(x_i)}^* 
                       & \text{if $\overline{a},\, \overline{b} \neq 0$ in $A$
                       and  $ab = x_i\cdots x_{i+s-1}$},
                       \\
                    0 
                    	&{\rm otherwise},\\
                \end{cases} %\label{Theta}
\end{align*}
where $a, \,b$ are paths in $\Delta$,
$m$ denotes the length of $ab$.
Moreover, 
$\sum_{i=1}^{s}\alpha_i$ is a $2$-cocycle and
the cohomology class $[\sum_{i=1}^{s}\alpha_i]$ is a basis of $D(HH_{2,\,s,\,{\overline{\gamma}} }(A))$.
We fix a nonzero element $k ({\neq}\,0) \in K$ and let $\alpha = k\sum_{i=1}^s \alpha_i$. 
Then we have the following proposition.
\begin{proposition}
	The ordinary quiver of $T_\alpha(A)$ coincides with $\Delta_{T_0(A)}$.
\end{proposition}
\begin{proof}
	We can prove this proposition by a similar way to \cite[Theorem 4.3]{Koie}.
\end{proof}

\subsection{Elementary cycles and $\alpha$-revived cycles}
Let $\alpha = k\sum_{i=1}^s \alpha_i$ be the $2$-cocycle defined in 
%same as in 
Section \ref{2-cocycle}.
We define an elementary  cycle and its weight for $T_{\alpha}(A)$ based on
\cite[Definition 3.1]{FP}. 
%\begin{definition}
	Let $C$ be an oriented cycle in $\Delta_{T_\alpha(A)}$.
	We say that $C$ is {\it elementary}
	if $C =  \delta_2 y_{p_i} \delta_1$
	for some paths $\delta_1$ and $\delta_2$ in $K\Delta$
	and $p_i \in \mathbb{M}$
	such that $\overline{p_i}^*(\overline{\delta_1\delta_2}) \neq 0$.
%	In this case, the {\it weight} of $C$ is
%	$w(C) := \overline{p_i}^*(\overline{\delta_1\delta_2}) \in K^*$.
%\end{definition}
%Moreover,
%we define $\alpha$-revived cycle.
%\begin{definition}
	Now let $C = a_{1} \cdots a_{j}$ be an oriented cycle in $\Delta_{T_\alpha(A)}$ where
	$a_1,\ldots, a_j \in \Delta_1$.
    We say that $C$ is {\it $\alpha$-revived}
	if 
    there exist $a,\,b \in \Delta_+$ such that
    $\overline{a},\,\overline{b} \neq 0$ in $A$,
    $C =a_1\cdots a_j = ab$
    and $\alpha(\overline{a},\,\overline{b}) \neq 0$. 
Then,
under the notation above,
it is easy to see that
$j = s$, $C = x_i\cdots x_{i+s-1}$ for some $i$
and $\alpha(\overline{a},\,\overline{b})(1_A) = k$,
where $k$ is the fixed element in the above.
Moreover, we define a {\it weight} $w(C)$ of an elementary cycle $C = \delta_2 y_{p_i} \delta_1$
by
$
				\overline{p_i}^*(\overline{\delta_1\delta_2})
$,
and we also define a {\it weight} $w(C)$ of an $\alpha$-revived cycle $C$
by $k$.

We say that a path $q$ is {\it contained}
in a path $q'$,
if $q' = \gamma_1 q \gamma_2$,
where $\gamma_1$, $\gamma_2$ are paths with
$t(\gamma_1) = s(q)$ and $s(\gamma_2) = t(q)$.

\begin{remark}[{cf. \cite[Remark 3.3]{FP}}]\label{Remark3.3}
	If $0 \neq \overline{v} \in A$, then there are paths $\delta_1,\delta_2$ in $K\Delta$
	and $p_j \in \mathbb{M}$ such that
	$\overline{p_j}^* (\overline{\delta_1 v \delta_2}) \neq 0$,
	and in particular, any nonzero path in $A$ is contained in an elementary cycle.
\end{remark}

\begin{remark}\label{Remark2}
	If $C = a_1 \cdots a_m$ with $a_1, \ldots, a_m \in (\Delta_{T_\alpha(A)})_1$
	is an elementary cycle,
	then $a_2a_3\cdots a_m a_1$ is also an elementary cycle. 
\end{remark}

\begin{remark}\label{Remark3}
	If $C = a_1 \cdots a_j$ with $a_1, \ldots, a_j \in \Delta_1$
	is an $\alpha$-revived cycle,
	then $a_2a_3\cdots a_j a_1$ is also an $\alpha$-revived cycle.
\end{remark}

\begin{definition}[{cf. \cite[Definition 3.4]{FP}}]
	Let $q$ be a path contained in an elementary cycle $C$ of
    length less than or equal to the length of $C$.
   The  {\it supplement} of $q$ in $C$ is defined as follows:
$$	
\begin{cases}
	\text{the trivial path $e_{s(q)}$}  
		&\text{if $s(q)=t(q)$,} \\
	\text{the path formed by the remaining
	arrows of $C$} 
		& \text{if $s(q) \neq t(q)$.}
\end{cases}
$$
%	If $s(q)=t(q)$, the {\it supplement} of $q$ in $C$
%	is the trivial path $e_{s(q)}$;
%    otherwise,
%    it is the path formed by the remaining
%	arrows of $C$.
\end{definition}

\subsection{Example}
We illustrate an example of the ordinary quiver of
a Hochschild extension algebra and we give some
examples of elementary cycles and $\alpha$-revived cycles.
\begin{example}\label{example1}
Let $\Delta$ be the following quiver:
$$
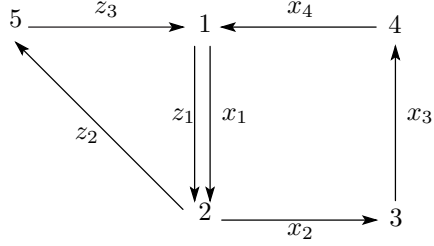
\input{ex1-1}
$$
and $A= K\Delta/R_\Delta^3$.
Let $\mathbb{M}$ be the set of paths of length $2$.
Then
$\mathbb{M}$ forms a $K$-basis of ${\rm soc}_{A^e}A$
of ${\rm dim}_K \,{\rm soc}_{A^e}A = 11$.
We put $\gamma = x_1x_2x_3x_4$.
For each $i$ $(1 \leq i \leq 4)$,
$\alpha_i : A \times A \longrightarrow D(A)$
corresponding to $\gamma$
is the map as follows:
$$
	\alpha_i(\overline{a},\,\overline{b}) 
    	= \begin{cases}
    		\overline{x_{i+3}}^*  & \text{if $\overline{a}, \overline{b} \neq 0$ in $A$ and $ab = x_ix_{i+1}x_{i+2}$}, \\
            \overline{e_{i+4}}^*  & \text{if $\overline{a}, \overline{b} \neq 0$ in $A$ and $ab = x_ix_{i+1}x_{i+2}x_{i+3}$}, \\
            0 & \text{otherwise}.
    	\end{cases}
$$
Let $k(\neq 0) \in K$ and $\alpha = k \sum_{i=1}^4 \alpha_i$.
Then, the ordinary quiver of $T_\alpha(A)$ coincides with $\Delta_{T_0(A)}$,
and $\Delta_{T_0(A)}$ is the following quiver:

%~\\
$$
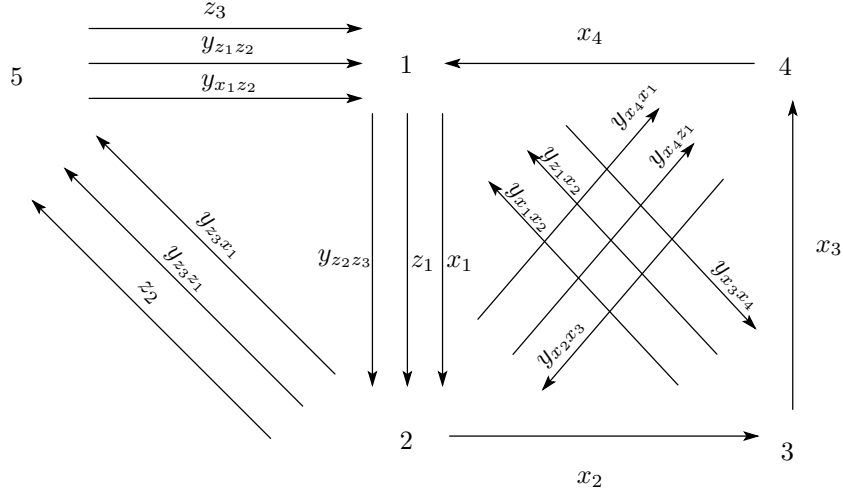
\input{ex1-2}
$$
%~\\

The elementary cycles with origin $1$ are 
\begin{align*}
	&y_{z_2z_3}z_2z_3,\quad    z_1y_{z_3z_1}z_3, \quad  x_1y_{z_3x_1}z_3,  \quad
		z_1z_2y_{z_1z_2}, \quad x_1z_2y_{x_1z_2}, \\
	&x_1x_2y_{x_1x_2},\quad z_1x_2y_{z_1x_2},\quad x_1y_{x_4x_1}x_4,\quad 
		z_1y_{x_4z_1}x_4,\quad y_{x_3x_4}x_3x_4.
\end{align*}
Moreover, there is only one $\alpha$-revived cycle with origin $1$, which is
$\gamma = x_1x_2x_3x_4$.

For the elementary cycle $y_{z_2z_3}z_2z_3$,
the supplement of $y_{z_2z_3}z_2$ in $y_{z_2z_3}z_2z_3$ is $z_3$
and the supplement of $z_2$ in $y_{z_2z_3}z_2z_3$ is $z_3y_{z_2z_3}$. 

\end{example}

%%%%%%%%%%%%%%%%%%%%%%%%%%%%%%%%%%%%%%%%%%%%%%%%%%%%%%5
%
%     section 3
%
%%%%%%%%%%%%%%%%%%%%%%%%%%%%%%%%%%%%%%%%%%%%%%%%%%%%%%%%%%
\section{Nonzero oriented cycles in $T_\alpha(A)$}
Let $T_\alpha(A)$ be a Hochschild extension algebra same as in Section $2$.
We note that
the ordinary quiver 
$\Delta_{T_\alpha(A)}$ coincides with $\Delta_{T_0(A)}$ which is given by
\begin{itemize}
 \item $(\Delta_{T_0(A)})_0 = \Delta_0$,
 \item $(\Delta_{T_0(A)})_1 = \Delta_1 \cup \{y_{p_1} ,\ldots, y_{p_t} \}$,
\end{itemize}
where, for each $i$, $y_{p_i}$ is an arrow from $t(p_i)$ to $s(p_i)$.

In this section, we give a characterization of nonzero oriented cycles in
$T_\alpha(A)$.
We will show that a nonzero oriented cycle is either
an elementary cycle or an $\alpha$-revived cycle.

Consider the homomorphism of $K$-algebras
$\Phi: K\Delta_{T_\alpha(A)} \longrightarrow T_\alpha(A)$
defined on the trivial paths and the arrows as follows:
\begin{align*}
	\Phi(e_i) 
    	&= (\overline{e_i} ,\, 0) 
			\text{\,for $i= 1, \ldots , n$,}  \\
    \Phi(a) 
    	&= (\overline{a} ,\,0)
        	\text{\,for $a \in \Delta_1$}, \\
    \Phi(y_{p_i}) 
    	&= (0,\, \overline{p_i}^*) 
              \text{\,for $p_i \in \mathbb{M}$.}
\end{align*}
Then $\Phi$ is surjective.
Let $\pi_1$ and $\pi_2$ are the projections induced by the decomposition
$T_{\alpha}(A) = A \oplus D(A)$,
and we define
%Associated with $\Phi$ are the morphisms
\begin{center}
	$\varphi_1 := \pi_1 \Phi : K\Delta_{T_\alpha(A)} \longrightarrow A$
    \quad and \quad
    $\varphi_2 := \pi_2 \Phi : K\Delta_{T_\alpha(A)} \longrightarrow D(A)$.
\end{center}

In this paper,
$Y$ denotes the ideal generated by the elements $y_{p_i}\,(i=1,\ldots,t)$ in $K\Delta_{T_\alpha(A)}$.

%For this purpose,
We show the following lemma.
\begin{lemma}\label{Lemma3.5}
    Suppose $z$ is a path in $K\Delta_{T_\alpha(A)}$.
    Then the following hold
    \begin{enumerate}[{\rm (1)}]
 		\item $\Phi(z)=0$ if $z$ contains two or more arrows $y_{p_i}(i= 1,\ldots,t)$. \label{(c)}
 		\item If $v=v_1 + v_2$ with a path $v_1 \in K\Delta$ and a path $v_2 \in Y$,
 					then
 					$$
 						\Phi(v) = \begin{cases} 
 						(\varphi_1(v_1),\,\varphi_2(v_2))
 						& \text{if  ${\rm length}\,v_1 < n$}, \\
 							(0,\,\varphi_2(v_1 + v_2))
 							& \text{if $n \leq {\rm length}\,v_1 \leq s$}, \\
 								(0,\,\varphi_2(v_2))
 								& \text{if $s < {\rm length}\,v_1$}.
 						\end{cases}
 					$$ \label{(a)}
% 		\item[(a-1)] If $v=v_1 + v_2$ with $v_1 \in K\Delta$, {\rm length}\,$v_1 < n$ and $v_2 \in Y$, 
%        				then $\Phi(v) = (\varphi_1(v_1),\,\varphi_2(v_2))$.
% 		\item[(a-2)] If $v=v_1 + v_2$ with $v_1 \in K\Delta$, $n \leq {\rm length}\,v_1 \leq q$ and $v_2 \in Y$,
%        				then $\Phi(v) = (0,\,\varphi_2(v_1 + v_2))$.
% 		\item[(a-3)] If $v=v_1 + v_2$ with $v_1 \in K\Delta$, $q < {\rm length}\,v_1$ and $v_2 \in Y$,
%        				then $\Phi(v) = (0,\,\varphi_2(v_2))$, taht is, $\varphi_2(v_1) = 0$.
 		\item $\varphi_2(z) \neq 0$ implies that $z \in Y$ or $n \leq {\rm length}\,z \leq s$. \label{(b)}
        \item Let $C$ be an elementary cycle or an $\alpha$-revived cycle with origin $e$.
        				Then, for any path $u(\neq e)$ in $K\Delta$, 
        				$\varphi_2(C)(\overline{e}) = w(C)$ 
                        and $\varphi_2(C)(\overline{u})=0$.\label{(g)}
        \item If $u$ is a path from $i$ to $j$ in $K\Delta$,
        				$v$ is a path from $j$ to $i$ in $K\Delta_{T_\alpha(A)}$
                        and both $uv$ and $vu$ are not $\alpha$-revived cycles,
                        then 
                        $
                        \varphi_2(v)(\overline{u}) 
                        	= \varphi_2(vu)(\overline{e_j})
                            	= \varphi_2(uv)(\overline{e_i}).
                        $\label{(e)}
        \item If  $u$ is a path from $i$ to $j$ in $K\Delta$,
        			 $v$ is a path from $j$ to $i$ in $K\Delta_{T_\alpha(A)}$ 
        				and both $uv$ and $vu$ are not $\alpha$-revived cycles,
        			  then $vu \in {\rm Ker}\,\Phi$ if and only if $uv \in {\rm Ker}\,\Phi$.\label{(i)}
        \item Let $u$ is a path in $K\Delta$ and $z$ a path in $Y$.
        			Then $\varphi_2(z)(\overline{u}) \neq 0$ implies
                    that $u$ is a supplement of $z$.\label{(d)}
	\end{enumerate}
\end{lemma}
\begin{proof}
	We get (\ref{(c)}) using that $D(A)^2 = 0$ in $T_\alpha(A)$.
    (\ref{(a)}) and (\ref{(b)}) follow directly from the definitions.
   
    (\ref{(g)})
    If $C$ is an elementary cycle,
    then (\ref{(g)}) follows from the definitions together with the hypothesis over $A$. 
    If $C$ is an $\alpha$-revived cycle,
    %By $n \leq s$, the length $l$ of an $\alpha$-revived cycle $C$
    %satisfy $n \leq l \leq 2n-1 \neq 2s-1$, that is, $l=s$.
    %So
    $C = ex_ix_{i+1} \cdots x_{i+s-1}$ for some $i \in \{1,\ldots s \}$.
    So we have
    \begin{align*}
    	\Phi (C) 
        	&= \Bigl( \overline{x_ix_{i+1}\cdots x_{i+n-1}},\, \alpha(\overline{x_i},\,\overline{x_{i+1}\cdots x_{i+n-1}}) \Bigr)
            (\overline{x_{i+n}},\,0) \cdots (\overline{x_{i+s-1}},\,0) \\
            	&= \Bigl( 0,\, \alpha \left( \overline{x_i},\,\overline{x_{i+1} \cdots x_{i+n-1}} \right) \Bigr)
                \left( \overline{x_{i+n}},\,0 \right) \cdots \left( \overline{x_{i+s-1}},\,0 \right) \\
                	&= \Bigl(0,\, \alpha(\overline{x_{i}},\,\overline{x_{i+1} \cdots x_{i+n-1}})\cdot
                    (\overline{x_{i+n} \cdots x_{i+s-1}}) \Bigr)
    \end{align*}
%\begin{align*}
%\left(\left(a+b\right)^2-\left(c+d\right)^2\right)
%\end{align*} 
%$$
%a \left[
%b \left\{
%\frac{1}{2} \left(
%c ( d + e )^x
%\right)
%\right\}^{\frac{y}{z}}
%\right]
%$$
%$$
% \left\{ x \left( y^* \right) \right\} 
%$$
%$
%\left(
%\begin{array}{ccccc}
%a_{11} & \cdots & a_{1i} & \cdots & a_{1n}\\
%\vdots & \ddots &        &        & \vdots \\
%a_{i1} &        & a_{ii} &        & a_{in} \\
%\vdots &        &        & \ddots & \vdots \\
%a_{n1} & \cdots & a_{ni} & \cdots & a_{nn}
%\end{array}
%\right)
%$
    and
    \begin{align*}
    	\varphi_2(C)(\overline{e})
        	&= \Bigl( \alpha(\overline{x_i},\,\overline{x_{i+1}\cdots x_{i+n-1}})\cdot
            (\overline{x_{i+n}\cdots x_{i+s-1}}) \Bigr) (\overline{e}) \\
            	&= w(C)\overline{x_{i+n}\cdots x_{i+s-1}}^*
                (\overline{ex_{i+n}\cdots x_{i+s-1}}) \\
                	&= w(C).
    \end{align*}
    It is clear that $\varphi_2(C)(\overline{u}) = 0$
    for any path $u$.
    
    (\ref{(e)}) %\underline{Case1: $\overline{u} = 0$ in $A$} \\
    First, we consider the case $\overline{u} = 0$ in $A$.
    We have $\varphi_2(v)(\overline{u}) = 0$.
    On the other hand, since
    $$
    		\Phi(u) = \begin{cases}
    			(\overline{u}, \, 0)  &\text{if ${\rm length}\,u < n$}, \\
    			(0,\, \varphi_2(u) )  & \text{if $n \leq {\rm length}\,u \leq q$}, \\
    					0 & \text{otherwise,}
	\end{cases}
    $$
    we have
    \begin{align*}
    		\varphi_2(vu)(\overline{e_j}) &= 
    		\pi_2 \Bigl( \Phi(v)\Phi(u) \Bigr)(\overline{e_j}) \\ &= 
    		\begin{cases}
				\pi_2\Bigl(\Phi(v) (\overline{u},\, 0) \Bigr) (\overline{e_j})  &\text{if ${\rm length}\,u < n$}, \\
					\pi_2\Bigl(\Phi(v) (0,\, \varphi_2(u))\Bigr)(\overline{e_j}) &\text{if $n \leq {\rm length}\,u \leq s$}, \\
						0 & \text{otherwise}.
			\end{cases}
   \end{align*}
    If ${\rm length}\,u \leq n$
    then $\pi_2(\Phi(v) (\overline{u},\, 0))(\overline{e_j}) = 0$ by $\overline{u} = 0$ in $A$.
    So we consider the case that $n \leq {\rm length}\,u \leq s$.
    We put $m = {\rm length}\,u$. Then $u = a_1a_2 \cdots a_m$ for some arrows
    $a_1, a_2, \ldots a_m \in \Delta$,
    and we have
    \begin{align*}
		&\varphi_2(u) \\
		&=
			\pi_2\Bigl((\overline{a_1},\,0)(\overline{a_2},\,0)
			\cdots(\overline{a_n},\,0)(\overline{a_{n+1}},\,0)
			\cdots(\overline{a_m},\,0)\Bigr) \\
				&= \pi_2\Bigl(\bigl(\overline{a_1a_2\cdots a_n},\, \alpha(\overline{a_1a_2\cdots a_{n-1}}\overline{a_n})\bigr)
				(\overline{a_{n+1}},\,0) \cdots (\overline{a_{m}},\,0)\Bigr) \\
					&= \begin{cases}
							\pi_2\Bigl( (0,\, \overline{x_{i+n}\cdots x_{i+s-1}}^*)(\overline{a_{n+1}},\,0) \cdots (\overline{a_{m}},\,0)\Bigr) 
							& \text{if $a_1\cdots a_n = x_i \cdots x_{i+n-1}$} \\ 
							& \quad \text{for some $i \in \{1,\ldots,s \}$}, \\
								\pi_2\Bigl((0,\,0)(\overline{a_{n+1}},\,0) \cdots (\overline{a_{m}},\,0)\Bigr) 
								& \text{otherwise}
					\end{cases} \\
						&= \begin{cases}
							\pi_2\Bigl(0,\, \overline{x_{i+n}\cdots x_{i+s-1}}^* \cdot \overline{a_{n+1} \cdots a_m}\Bigr)
							 & \text{if $a_1\cdots a_n = x_i \cdots x_{i+n-1}$} \\ 
							& \quad \text{for some $i \in \{1,\ldots,s \}$}, \\
								0 & \text{otherwise}
						\end{cases} \\
							&= \begin{cases}
									\overline{x_{i+n}\cdots x_{i+s-1}}^*\cdot \overline{a_{n+1} \cdots a_m}
							 		& \text{if $a_1\cdots a_n = x_i \cdots x_{i+n-1}$} \\ 
										& \quad \text{for some $i \in \{1,\ldots,s \}$}, \\
										0 & \text{otherwise}.
							\end{cases}
	\end{align*}
	On the other hand, $\Phi(v) = (\overline{v},\,0)$ or $\Phi(v) = (0,\, f)$ for some $f \in D(A)$.
	So we have
	\begin{align*}
		&\pi_2\Bigl(\Phi(v) (0,\, \varphi_2(u))\Bigr)(\overline{e_j}) \\
			&= \begin{cases}
				\pi_2\Bigl((\overline{v},\,0) (0,\, \varphi_2(u)) \Bigr) (\overline{e_j})
				& \text{if $\Phi(v) = (\overline{v},\,0)$}, \\
					0 & \text{otherwise}
			\end{cases} \\
				&= \begin{cases}
					\pi_2\Bigl(
						0,\,
						\overline{v} \cdot \overline{x_{i+n}\cdots x_{i+s-1}}^* \cdot
						\overline{a_{n+1} \cdots a_m}\Bigr)(\overline{e_j}
					)
					& \text{if $\Phi(v) = (\overline{v},\,0)$ and} \\
					& \quad \text{$a_1\cdots a_n = x_i \cdots x_{i+n-1}$} \\ 
					& \quad \text{for some $i \in \{1,\ldots,s \}$}, \\
						0 & \text{otherwise}
				\end{cases} \\
					&= \begin{cases}
						\overline{x_{i+n}\cdots x_{i+s-1}}^*
						(\overline{a_{n+1} \cdots a_m v})
						& \text{if $\Phi(v) = (\overline{v},\,0)$ and} \\
						& \quad \text{$a_1\cdots a_n = x_i \cdots x_{i+n-1}$} \\ 
						& \quad \text{for some $i \in \{1,\ldots,s \}$}, \\
							0 & \text{otherwise}.
					\end{cases}•
	\end{align*}
	If $\overline{x_{i+n}\cdots x_{i+s-1}}^*(\overline{a_{n+1} \cdots a_m v}) \neq 0$,
	then $uv = a_1 \cdots a_m v = x_i \cdots x_{i+s-1}$.
	This contradicts that $uv$ is not an $\alpha$-revived cycle. 
%    aaaaaaaaaaaaaaaaaaaaaaaaaaaaaaaaaaaaaaaaaaaaaaaaaaa\\
%    If $u$ is not $\alpha$-revived cycle, then
%    $$
%    \pi_2(\Phi(v)\Phi(u))(\overline{e_j}) = \pi(\Phi(v)(\overline{u},\,0))(\overline{e_j}) = 0,
%    $$
%    and (e) follows.
%    If $u$ is $\alpha$-revived cycle, that is, $i=j$,
%    then $\Phi(u) = (0,\, \overline{e_i}^*)$ and
%    \begin{equation}\label{e-1}
%    	\pi_2(\Phi(v)\Phi(u))(\overline{e_j}) = \pi_2(\Phi(v)(0,\, \overline{e_i}^*))(\overline{e_i}). 
%    \end{equation}
%    Since $v$ is a path, $\Phi(v)$ is $(a,\,0)$ or $(0,\,f)$
%    for some $a \in A$ or $f \in D(A)$.
%    If $\Phi(v) = (0,\,f)$, then (\ref{e-1}) is equal to $0$.
%    If $\Phi(v) = (a,\,0)$, then it is clear that $a = \overline{v}$ and
%    (\ref{e-1}) is equal to
%    \begin{align*}
%    	\pi_2(0,\, \overline{v}\overline{e_i}^*)(\overline{e_i})
%        	= (\overline{v}\overline{e_i}^*)(\overline{e_i})
%            	= \overline{e_i}^*(\overline{e_i}\overline{v})
%                	= \begin{cases}
%                    	1  & (\overline{v} = \overline{e_i}) \\
%                        0  & otherwise.
%                    \end{cases}
%    \end{align*}
%	By the hypothesis that $uv$ is not $\alpha$-revived, (\ref{e-1}) is zero. \\
%	aaaaaaaaaaaaaaaaaaaaaaaaaaaaaaaaaaaaaaaaaaaaaaaaaaa
	
	Next we consider the case $\overline{u} \neq 0$ in $A$.
	%\underline{Case2: $\overline{u} \neq 0$ in $A$} \\
    Since $v$ is a path,
    by Lemma \ref{Lemma3.5} (\ref{(a)}), 
    $\Phi(v) = (\overline{v},\,0)$ or $\Phi(v) = (0,\,f)$
    for some $f \in D(A)$.
    Since $vu$ is not an $\alpha$-revived cycle,
    if $\Phi(v) = (\overline{v},\, 0)$,
    then we have $\alpha(\overline{v},\, \overline{u}) = 0$ and
    \begin{align*}
    	\varphi_2(vu)(\overline{e_j})
        	= \pi_2\Bigl((\overline{v},\,0)(\overline{u},\,0)\Bigr)(\overline{e_j}) 
            	= \pi_2\Bigl(\overline{vu}, \alpha(\overline{v},\overline{u})\Bigr)(\overline{e_j}) 
                	= \pi_2(\overline{vu},\,0)(\overline{e_j}) 
                    	=0.
    \end{align*}
    On the other hand,
    $\varphi_2(v)(\overline{u}) = \pi_2(\overline{v},\,0)(\overline{u}) = 0$.
	If $\Phi(v) = (0,\,f)$, then we have
    \begin{align*}
    	\varphi_2(vu)(\overline{e_j}) 
        	= \pi_2\Bigl((0,\,f)(\overline{u},\,0)\Bigr)(\overline{e_j})
        		= (f\overline{u})(\overline{e_j})
            			= \pi_2(0,\,f)(\overline{u})
                    		= \varphi_2(v)(\overline{u}).
    \end{align*}
	Similarly, $\varphi_2(v)(\overline{u}) = \varphi_2(vu)(\overline{e_j})$.
    
	(\ref{(i)}) Since $vu$ and $uv$ are not $\alpha$-revived cycles,
    it suffices to consider the case $vu,\,uv \in Y$.
    By the assumption that any oriented cycle in $K\Delta$ is zero in $A$,
    for any path $x$ in $\Delta$,
    $$
    \varphi_2(vu)(\overline{x}) = \begin{cases}
							\varphi_2(vu)(\overline{e_j})  & \text{if $\overline{x}= \overline{e_j}$}, \\
							0 & \text{otherwise},
						\end{cases}
    $$
    and
    $$
    \varphi_2(uv)(\overline{x}) = \begin{cases}
							\varphi_2(uv)(\overline{e_i})  & \text{if $\overline{x}= \overline{e_i}$}, \\
							0 & \text{otherwise}.
						\end{cases}
    $$
    If $vu \in {\rm Ker}\,\Phi$,
    by (\ref{(e)}), then
    $\varphi_2(uv)(\overline{e_i}) = \varphi_2(vu)(\overline{e_j}) = 0$.
    So we have $\varphi_2(uv) = 0$, and $vu \in {\rm Ker}\,\Phi$ implies that
    $uv \in {\rm Ker}\,\Phi$.
    The converse holds, similarly.  
    
    (\ref{(d)}) By (\ref{(c)}), $z=\delta_2 y_{p_i} \delta_1$ with paths $\delta_1$, $\delta_2$ in $K\Delta$
    and $i \in \{ 1, \ldots, t \}$.
    Then $\overline{p_i}^*(\overline{\delta_1 u \delta_2}) = \varphi(z)(\overline{u}) \neq 0$.
    Therefore $u$ is a supplement of $z$ in the elementary cycle $\delta_2 y_{p_i} \delta_1 u$.
\end{proof}

\begin{theorem}\label{Corollary3.8}
	Let $C$ be an oriented cycle in $K\Delta_{T_\alpha(A)}$.
    Then the following conditions are equivalent:
    \begin{enumerate}[{\rm (i)}]
 		\item $C$ is an elementary cycle or $\alpha$-revived cycle.
        \item $C$ is nonzero in $T_\alpha(A)$.
    \end{enumerate}
\end{theorem}

\begin{proof}
	By Lemma \ref{Lemma3.5} (\ref{(g)}),
	(i) implies (ii).	
	We will show that (ii) implies (i).
	Let $C_1$ be a cycle in $K\Delta$,
	that is,
	$C_1 = z_1 \cdots z_m$ for some $m \geq n$ and $z_1, \ldots, z_m \in \Delta_1$.
	We  assume that $\varphi_2(C_1) \neq 0$ and we will show that
	$C_1$ is an $\alpha$-revived cycle.
	By Lemma \ref{Lemma3.5} (\ref{(a)}),
	$\Phi(C_1) = (0,\, \varphi_2(C_1))$.
	By Lemma \ref{Lemma3.5} (\ref{(a)}),
	we have $n \leq m \leq s$ and 
	\begin{align*}
		0 &\neq \varphi_2(C_1) \\
			&= \pi_2\Bigl((\overline{z_1},\, 0)\cdots (\overline{z_m},\, 0)\Bigr) \\
				&= \pi_2 \Bigl(
						(0,\, \alpha(\overline{z_1\cdots z_{n-1}},\,\overline{z_n}))
						(\overline{z_{n+1}},\, 0)
						\cdots
						(\overline{z_{m}},\, 0)
					\Bigr)  \\
					&= \alpha(\overline{z_1\cdots z_{n-1}} ,\, \overline{z_n}) \cdot
					\overline{z_{n+1} \cdots z_m}.
	\end{align*}
	So $\alpha(\overline{z_1\cdots z_{n-1}} ,\, \overline{z_n}) \neq 0$,
	we have 
	$z_1 \cdots z_n$ coincides with $x_i \cdots x_{i+n-1}$
	and
	$
	  \alpha(\overline{z_1 \cdots z_{n-1}},\, \overline{z_n}) 
	    = (\overline{x_{i+n} \cdots x_{i+s-1}})^*
	$
	for some $i \in \{ 1,\ldots, s\}$.
	Since
	$$
	  0 \neq
	    \alpha(\overline{z_1\cdots z_{n-1}} ,\, \overline{z_n})\cdot \overline{z_{n+1} \cdots z_m}
	      = \overline{x_{i+n} \cdots x_{i+s-1}}^* \cdot \overline{z_{n+1} \cdots z_{n}},
	$$
	there exists a path $\delta$ in $\Delta$ such that
	$$
	  0 \neq (\overline{x_{i+n} \cdots x_{i+s-1}}^* \cdot \overline{z_{n+1} \cdots z_{n}})(\overline{\delta})
	    = \overline{x_{i+n} \cdots x_{i+s-1}}^*(\overline{z_{n+1} \cdots z_{n}}\overline{\delta}).
	$$
	Then $z_{n+1} \cdots z_m \delta = x_{i+n} \cdots x_{i+s-1}$
	and $z_1 \cdots z_m \delta$ coincides with an $\alpha$-revived cycle
	$x_{i} \cdots x_{i+n-1} x_{i+n} \cdots x_{i+s-1}$.
	Therefore, by that $C_1 = z_1 \cdots z_m$ is a cycle, $\delta$ is a trivial path
	and $C_1$ is an $\alpha$-revived cycle.	
%	aaaaaaaaaaaaaaaaaaaaaaaaaaaaaaaaaaaaaaaaaaaaaaaaaaaa
%	Let $C_1$ be a cycle from $j$ to $j$ in $K\Delta$,
%    that is, $C_1 = z_1 \cdots z_m$ for some
%    $z_1, \ldots , z_m \in \Delta_1$ and $m \geq n$.
%    Then $\Phi(C_1) = (0,\, \varphi_2(C_1))$,
%    by Lemma \ref{Lemma3.5}(a-2) and (a-3).
%    If $\varphi_2(C_1) \neq 0$, we have
%    $$
%    	0 \neq \varphi_2(C_1)
%        	= \pi_2\Phi(0,\, \alpha(z_1,\,z_2 \ldots z_m))
%            	= \alpha(z_1,\,z_2 \ldots z_m)
%    $$
%    and $C_1$ is $\alpha$-revived cycle.
%    aaaaaaaaaaaaaaaaaaaaaaaaaaaaaaaaaaaaaaaaaaaaaaaaaaaaaaaa
    
    Next, let $C_2$ be a cycle from $j$ to $j$ in $Y$.
     We assume that $\Phi(C_2) \neq 0$.
    We have $\Phi(C_2) = (0,\, \varphi_2(C_2))$,
    by Lemma \ref{Lemma3.5} (\ref{(a)}).
    By Lemma \ref{Lemma3.5} (\ref{(c)}),
    suppose now that $C_2$ contains exactly one arrow $y_{p_i}$.
    Then $C_2 = \delta_2 y_{p_i} \delta_1$
    with $\delta_1$, $\delta_2$ paths in $K\Delta$.
    Since $\varphi_2(C_2) \neq 0$, there exists a path $u$
    in $K\Delta$ such that $\varphi_2(C_2)(\overline{u}) \neq 0$.
    Then $u$ is a path from $j$ to $j$ and $\overline{u} \neq 0$.
    So, by the hypothesis on $A$, $u = e_j$ and
    thus $C_2$
    is an elementary cycle.
    
    Therefore, 
    any nonzero oriented cycle is elementary or $\alpha$-revived. 
    \end{proof}
    
\begin{proposition}\label{Proposition3.6}
    For each $j \in (\Delta_{T_\alpha(A)})_0$,
    let $I_j$ be the ideal in $K\Delta_{T_\alpha(A)}$ generated by
    \begin{enumerate}
 		\item[{\rm (i)}] oriented cycles in $K\Delta_{T_\alpha(A)}$ from $j$ to $j$
        		which are neither elementary nor $\alpha$-revived,
 		\item[{\rm (ii)}] elements $w(C')C - w(C)C' \in K\Delta_{T_\alpha(A)}$,
        		where $C$, $C'$ are elementary or $\alpha$-revived cycles
                with origin $j$.
	\end{enumerate}
	Then ${\rm Ker}\,\Phi \cap e_j K\Delta_{T_\alpha(A)}e_j$ generates $I_j$.
\end{proposition}

\begin{proof}
	By the proof of Theorem \ref{Corollary3.8},
	oriented cycles from $j$ to $j$
    which are neither elementary nor $\alpha$-revived 
    are in ${\rm Ker}\,\Phi \cap e_j K\Delta_{T_\alpha(A)}e_j$.
    Let now $z = w(C') C - w(C) C'$
    be an element as defined in (ii).
    Then 
    $\Phi(z) = (0,\, \varphi_2(z))$, and
    $\varphi_2(z) = w(C') \varphi_2(C) - w(C)\varphi_2(C')$.
    By Lemma \ref{Lemma3.5} (\ref{(g)}), for any path $u$ in $K\Delta$,
    \begin{align*}
    	\varphi_2(z)(\overline{u})
        	&= (w(C') \varphi_2(C) )(\overline{u})
            - (w(C)\varphi_2(C'))(\overline{u}) \\
            	&= \begin{cases}
                	w(C')w(C) - w(C)w(C') & \text{if $u= e_j$}, \\
                    0   & \text{otherwise}
                \end{cases} \\
                	&= 0.
    \end{align*}
    Hence, $z \in {\rm Ker}\,\Phi \cap e_j K\Delta_{T_\alpha(A)}e_j$.
    
    Thus $I_j$ is contained in the ideal generated by
    ${\rm Ker}\,\Phi \cap e_j K\Delta_{T_\alpha(A)}e_j$.
    The proof of the other inclusion is similar to the proof of
    \cite[Proposition 3.6]{FP}.
\end{proof}

%%%%%%%%%%%%%%%%%%%%%%%%%%%%%%%%%%%%%%%%%%%%%%%%%%%%%%%%%%%%%%%%%%
%
%     Section 3
%
%%%%%%%%%%%%%%%%%%%%%%%%%%%%%%%%%%%%%%%%%%%%%%%%%%%%%%%%%%%%%%%%%%%%
\section{An application of a theorem of Brenner}
In this section, we give the number of indecomposable direct summands of
the middle term of almost split sequence for $T_\alpha(A)$.
We define a relation on the set of nonzero oriented cycles
with same origin in $\Delta_{T_\alpha(A)}$.
We will show that the above number is equal to the cardinality of the equivalence classes. 

The following proposition is a generalization of Remark \ref{Remark3.3}.
\begin{proposition}\label{path is in a nonzero oriented cycle}
	Any nonzero path in $T_\alpha(A)$ is
	contained in a nonzero oriented cycle.
\end{proposition}

\begin{proof}
	Let $w$ be a nonzero path in $T_\alpha(A)$.
	Then $w$ is in $K\Delta$ or in $Y$. 
	First, we suppose that $w \in K\Delta$.
	Since $\Phi(w) \neq 0$, 
	we have ${\rm length}\, w \leq s$
	by Lemma \ref{Lemma3.5} (\ref{(a)}).
    If $n \leq {\rm length}\, w \leq s$,
    by the construction of $\alpha$
    and the same calculation as in the proof of Lemma \ref{Lemma3.5} (\ref{(e)}),
    $w$ is contained in an $\alpha$-revived cycle.
    If ${\rm length}\, w < n$, then $w \neq 0$ in $A$ and
    $w$ is in an elementary cycle by Remark \ref{Remark3.3}.
    
    Next, we suppose that $w \in Y$.
    We have  $\varphi_2(w) \neq 0$.
    Hence $\varphi_2(w)(\overline{u}) \neq 0$ for some path $u$ in $K\Delta$.
    By Lemma \ref{Lemma3.5} (\ref{(d)}), $u$ is a supplement of $w$,
    so $w$ is contained in an elementary cycle $wu$.
\end{proof}
\begin{definition}
	For each $h \in (\Delta_{T_\alpha(A)})_0$,
	let us denote by $\mathscr{C}_h$
    the set of all oriented cycles $C$ such that
    $C \neq 0$ in $T_\alpha(A)$
    and $s(C) = t(C) = h$.
    Let $C,\, C'$ be in $\mathscr{C}_h$.
    If there exists an arrow $a$ belonging to
    $C$ and $C'$ with
    $s(a) = h$ or $t(a) = h$,
    then we write $C \mathscr{R} C'$.
\end{definition}

\begin{definition}
	For each $h \in (\Delta_{T_\alpha(A)})_0$,
	let $\mathscr{A}_h = \{ a \in (\Delta_{T_\alpha(A)})_1 \mid t(a)=h \}$.
	For
	$a,\, a' \in \mathscr{A}_h$,
	 if there exists an arrow $b \in (\Delta_{T_\alpha(A)})_1$
    such that $ab \neq 0$ and $a'b \neq 0$ in $T_\alpha(A)$
    then we write
	$a \mathscr{R}' a'$.
\end{definition}

We note that, for any path $a \in \mathscr{A}_h$, $a \mathscr{R}' a$ holds by
Proposition \ref{path is in a nonzero oriented cycle}.

From now on, we denote by $``\equiv"$ and $``\approx"$
the equivalence relations generated by $\mathscr{R}$ in $\mathscr{C}_h$
and by $\mathscr{R}'$ in $\mathscr{A}_h$, respectively.

We next want to give the precise connection between these equivalence
relations. For this purpose, the following results will be useful.

\begin{proposition}\label{Cor4.8}
	${\rm card}(\mathscr{C}_h/{\equiv}) = {\rm card}(\mathscr{A}_h/{\approx})$.
\end{proposition}

\begin{proof}
For $C \in \mathscr{C}_h$,
we denote by $a_C$ the last arrow of $C$,
that is,
$C = \delta a h$ for some path $\delta$.
Consider the map
$u : \mathscr{C}_h/{\equiv} \longrightarrow \mathscr{A}_h/{\approx}$
defined by $u(\overline{C}) = \overline{a_C}$
for $C \in \mathscr{C}_h$.
We will show that $u$ is well-defined.
It suffices to show that if $C_1 \mathscr{R} C_2$
then $a_{C_1} \mathscr{R'}a_{C_2}$.
We assume
$C_1 \mathscr{R} C_2$.
Then $a_{C_1} = a_{C_2}$ or
$C_i = b \delta_i a_{C_i}$ for some arrow $b$ and some path $\delta_i$ $(i = 1,\,2)$.
If $a_{C_1} = a_{C_2}$ then $a_{C_1} \mathscr{R'}a_{C_2}$, clearly.
If $C_i = b \delta_i a_{C_i}$,
by Theorem \ref{Corollary3.8},
then $C_i$ is an elementary cycle or an $\alpha$-revived cycle (i=1,\,2).
Thus,
by Remark \ref{Remark2} and Remark \ref{Remark3},
$a_{C_i}\delta_i b \neq 0$
and both $a_1b$ and $a_2b$ are not zero.
So we have $a_{C_1} \mathscr{R'} a_{C_2}$.
And $u$ is well-defined.
On the other hand,
for $a \in \mathscr{A}_h$,
by Proposition \ref{path is in a nonzero oriented cycle},
there is a cycle $\delta_1 ah \delta_2(\neq 0)$ for some path $\delta_1$
and $\delta_2$.
By Remark \ref{Remark2} and Remark \ref{Remark3},
$\delta_2 \delta_1 ah \neq 0$.
We denote $C_a :=\delta_2\delta_2 ah$.
Then $\overline{C_a}$ does not depend on the choice
of $\delta_1$ and $\delta_2$,
that is,
$\overline{C_a}$ is uniquely determined. 
We define the map
$v : \mathscr{A}_h/{\approx}\, \longrightarrow \mathscr{C}_h/{\equiv}$
by $v(\overline{a}) = \overline{C_a}$.
We will show that $v$ is well-defined.
If $a_1 \mathscr{R'} a_2$ then there exists a path $b$ such that
$a_1b \neq 0$ and $a_2b \neq 0$.
By Proposition \ref{path is in a nonzero oriented cycle},
there is a cycle $\gamma_i a_i b \gamma'_i$
for some path $\gamma_i$ and $\gamma'_i$ $(i = 1,\,2)$.
Thus we have
$$
	C_{a_1} 
		\mathscr{R}\, (b \gamma'_1 \gamma_1 a_1)
			\mathscr{R}\, ( b \gamma'_2 \gamma_2 a_2)
				\mathscr{R}\, C_{a_2}.
$$
So $C_{a_1} \equiv C_{a_2}$.
It is easy to see that $u$ and $v$ are mutually
inverse maps.
\end{proof}

%To prove the main theorem, we show the following lemma.t
%\begin{lemma}\label{Cor3.10}
%	Let $q$ be an arrow from $i$ to $j$ in $K\Delta_{T_\alpha(A)}$.
%    Suppose that either $\overline{rq} = 0$ or $\overline{qr} = 0$
%    in $T_\alpha(A)$ for the supplement $r$ of $q$.
%    Then $\overline{q} = 0$.
%\end{lemma}
%
%\begin{proof}
%	Suppose first that $\overline{rq} = 0$. \\
%    \underline{Case1: $q \in K\Delta$} 
%    If $\overline{q} \neq 0$, then we know by Remark \ref{Remark3.3}
%    that there exist $\delta_1,\, \delta_2 \in K\Delta$ and $p_i \in \mathbb{M}$
%    such that $\overline{p_i}^*(\overline{\delta_1 q \delta_2}) = 0$,
%    and thus, $C = y_{p_i} \delta_1 q \delta_2$ is an elementary cycle
%    and $r = \delta_2 y_{p_i} \delta_1$.
%    Hence we have
%    $$
%    	\varphi_2(r)(\overline{q}) 
%        	= (\overline{\delta_2}\, \overline{p_i}^* \overline{\delta_1})(\overline{q})
%            	= \overline{p_i}^*(\overline{\delta_1 q \delta_2})
%                	\neq 0.
%    $$
%    Since $C$ is a not $\alpha$-revived cycle, by Lemma \ref{Lemma3.5} (\ref{(e)}),
%    $\varphi_2(rq) \neq 0$ i.e. $\Phi(rq) \neq 0$,
%    contradicting the assumption that $\overline{rq}=0$ in $T_\alpha(A)$.
%    Thus we have $\overline{q}=0$, as desired.
%    
%    \underline{Case2: $q \in Y$} 
%    It is similar to case 1.
%    
%    This ends the proof, since we knew by Lemma \ref{Lemma3.5} (\ref{(i)}) that
%    $\overline{qr} =0$ if and only if $\overline{rq} =0$.
%\end{proof}

We have the  following theorem, which is similar to \cite[Proposition 4.9]{FP}: 
\begin{proposition}
	Let $h$ be a vertex in $\Delta_{T_\alpha(A)}$,
    and let $e_h$ be the idempotent element corresponding to $h$. 
    Then we have $N_{e_h}= n_{e_h} = {\rm card}(\mathscr{C}_h/{\equiv})$.
\end{proposition}
\begin{proof}
The proof is similar to \cite[Proposition 4.9]{FP}.
Let $\Lambda_1, \ldots, \Lambda_t$ be the equivalence classes of $\mathscr{A}_h$ by $\approx$. 
We put 
$\Gamma_i = \{ \text{$b$\,:\,arrow}  \mid \text{$\exists a \in \Lambda_i$ s.t. $ab \neq 0$} \}$
for each $i (1 \leq i \leq t)$,
and we set $\Lambda_0 = \Gamma_0 = \emptyset$.
We note that
$\Gamma_i \neq \emptyset$ for each $i = 1,\ldots,t$.
By construction, the pair $(t,\,t)$ is in $\mathscr{N}$.
We shall prove that  if $(N,\,n) \in \mathscr{N}$,
then $N-n \geq t$.
In fact, let $\Lambda'_i$ and $\Gamma'_i$ be sets of arrows satisfying Brenner's condition,
for $0 \leq i \leq n$.
Using the above conditions,
we conclude that
\begin{enumerate}
	\item $\Gamma'_i = \{b \mid \text{there exists $a \in \Lambda'_i$ with $ab \neq 0$ }  \}$
				for $i=1,\ldots,n$,
	\item $\Gamma'_0 = \Lambda'_0 = \emptyset$.
\end{enumerate}
It follows from (2) that $N=n$.
On the other hand, it is easy to see that if $a \approx a'$
then there exists $j$ such that $a,\, a' \in \Lambda'_j$.
Then, for $1 \leq i \leq t$,
there exists $j$ such that
$\Lambda_i \subset \Lambda'_{j}$,
which implies $n \leq t$.
Therefore,
$
  N_{e_h} = n_{e_h} = {\rm card}(\mathscr{A}_h/ \approx)={\rm card}(\mathscr{C}_h/{\equiv})
$
by Proposition \ref{Cor4.8}. 
\end{proof}

The following theorems are partial generalizations of \cite{FP}.

\begin{theorem}\label{cor1}
	Let $S_h$ be the simple $T_\alpha(A)$-module corresponding to the
    vertex $h$.
    Then the number of indecomposable direct summands of the middle
    term of almost split sequence
    $$
    	0 \longrightarrow S_h \longrightarrow E \longrightarrow \tau^{-1}S_h \longrightarrow 0
    $$
    is equal to the number of equivalence classes in $\mathscr{C}_h$.
    Furthermore, the number of indecomposable projective summands of
    $E$ is equal to zero.
\end{theorem}

\begin{theorem}
	Let $P_h$ be the indecomposable projective $T_\alpha(A)$-module
    corresponding to the vertex $h$.
    Then the number of indecomposable direct summands of
    ${\rm rad}\,P_h/{\rm soc}\,P_h$ is equal to the number of equivalence classes
    in $\mathscr{C}_h$.
\end{theorem}

\begin{corollary}
	Let $n \geq 3$ and $h \in \Delta_0$ be neither sink nor source in $\Delta$.
	Then we have ${\rm card}(\mathscr{C}_h/{\equiv}) = 1$.
\end{corollary}
%${\rm card}(\mathscr{C}_h/{\equiv}) = {\rm card}(\mathscr{A}_h/{\approx})$
\begin{proof}
By Proposition \ref{Cor4.8},
we have
$
	{\rm card}(\mathscr{C}_h/{\equiv}) = {\rm card}(\mathscr{A}_h/{\approx})
$.
So we show ${\rm card}(\mathscr{A}_h/{\approx})=1$.
Let 
$
	\mathscr{A}_h = \{ x_1,\, x_2, \ldots, x_m,\, y_1,\, y_2, \ldots, y_l \}
$,
where
$x_1, \ldots, x_m \in \Delta_0$
and
$y_1, \ldots, y_l \in (\Delta_{T_\alpha(A)})_0 \backslash \Delta_0$.
Now $h$ is not sink in $\Delta$,
so there exists an arrow $w \in \Delta_1$ such that $s(w)=h$.
By the assumption that $A=K\Delta/R_{\Delta}^n$ and $n \geq 3$,
we have $x_i w \neq 0$ in $A$ for any $1 \leq i \leq m$.
So $x_i w \neq 0$ in $T_\alpha(A)$ and 
$$
	x_1 \approx x_2 \approx \cdots \approx x_m 
$$
in $\mathscr{A}_h$.

Moreover, for any $j \in \{ 1,\ldots,l \}$,
by Proposition \ref{path is in a nonzero oriented cycle},
$y_j$ is contained in the $\alpha$-revived cycle.
Then we have $y_j v \neq 0$ in $T_\alpha(A)$ for some $v \in \Delta_1$.
And, by the same reason as above,
$x_1 v \neq 0$
in $T_\alpha(A)$.
So we have $x_1 \approx y_j$ in $\mathscr{A}_h$.
Hence we have ${\rm card}(\mathscr{C}_h/{\equiv}) = 1$.
\end{proof}

\begin{example}
Let $T_\alpha(A)$ be the same as in Example \ref{example1}.
Then, elements of $\mathscr{C}_1$ are elementary cycles and $\alpha$-revived
cycles in Example \ref{example1}. We have $\rm{card}( \mathscr{C}_1/ {\equiv}) = 1$.
On the other hand, $\mathscr{A}_1 = \{x_4,\, z_3,\, y_{z_1z_2},\, y_{x_1z_2} \}$. 
We have
\begin{align*}
	& x_4z_1 \neq 0 ,\quad z_3z_1\neq 0 ,\quad y_{z_1z_2} z_1 \neq 0, \\
 	& x_4x_1 \neq 0 ,\quad y_{x_1z_2} x_1 \neq 0,
\end{align*}
in $T_\alpha(A)$.
So all elements in $\mathscr{A}_1$ are equivalent, and 
${\rm card}(\mathscr{A}_h/ {\approx}) = 1$.
So, for almost split sequence
$$
	0 \longrightarrow S_1 \longrightarrow E \longrightarrow \tau^{-1}S_1 \longrightarrow 0,
$$
$E$ is a non-projective indecomposable module by Theorem \ref{cor1}.
\end{example}

\begin{example}
Let $\Delta$ be the following quiver:
$$
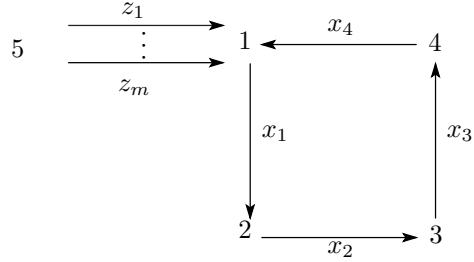
\input{ex2-1}
$$
and $A= K\Delta/R_\Delta^3$.
We put $\gamma = x_1x_2x_3x_4$.
Then, for each $i$ $(1 \leq i \leq 4)$,
$
\alpha_i : A \times A \longrightarrow D(A)
%(i=1,\ldots,4)
$ corresponding to $\gamma$
is the map as follows:
$$
	\alpha_i(\overline{a},\,\overline{b}) 
    	= \begin{cases}
    		\overline{x_{i+3}}^*  & \text{if $\overline{a}, \overline{b} \neq 0$ in $A$ and $ab = x_ix_{i+1}x_{i+2}$}, \\
            \overline{e_{i+4}}^*  & \text{if $\overline{a}, \overline{b} \neq 0$ in $A$ and $ab = x_ix_{i+1}x_{i+2}x_{i+3}$}, \\
            0 & \text{otherwise},
    	\end{cases}
$$
where the definition of $\alpha_i$'s is given in Section 2.
Let $k(\neq 0) \in K$ and $\alpha = k \sum_{i=1}^4 \alpha_i$.
Then, the ordinary quiver of $T_\alpha(A)$ coincides with $\Delta_{T_0(A)}$,
and $\Delta_{T_0(A)}$ is the following quiver:

%~\\
$$
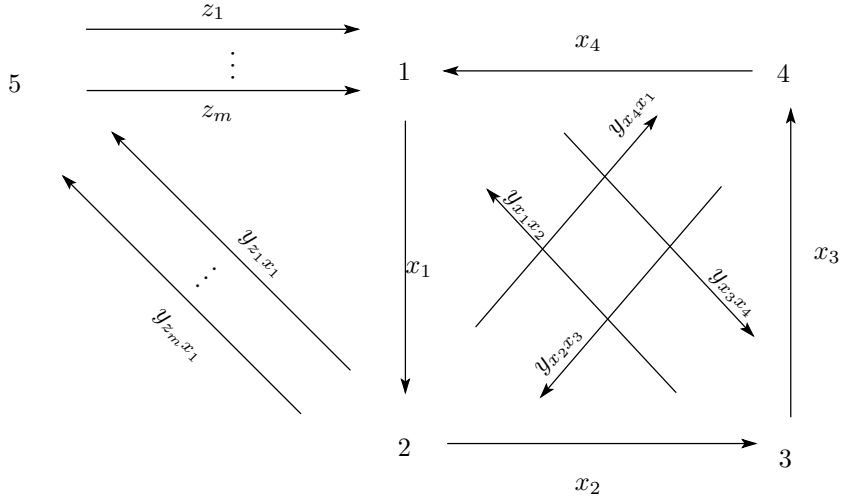
\input{ex2-2}
$$
~\\
In this case,
elementary cycles with origin $5$ are 
\begin{align*}
	&z_1x_1y_{z_1x_1},\, z_2x_1y_{z_2x_1},\ldots, z_mx_1y_{z_mx_1}.
\end{align*}
Moreover, there is no $\alpha$-revived cycle with origin $5$.
In this case,
if $i \neq j$ then
$z_ix_1y_{z_ix_1}$ and $z'_jx_1y_{z'_jx_1}$ are not equivalent in 
$\mathscr{C}_5$. And we have 
${\rm card}( \mathscr{C}_5/ {\equiv}) = m$.
\end{example}

\section*{Acknowledgment}
The author is grateful to Professor Katsunori Sanada and 
Doctor Tomohiro Itagaki
for many helpful comments and suggestions on improving the clarity of the article.

%\noindent
%{\bf
%	Acknowledgment
%}
%
%The authors thank Professor Takahiko Furuya for many valuable comments
%and suggestions.

%%%%%%%%%%%%%%%%%%%%%%%%%%%%%%%%%
%%%%%%%%% Reference %%%%%%%%%%%%%
%%%%%%%%%%%%%%%%%%%%%%%%%%%%%%%%%

\end{document}

%% file: ex1-1.tex
%WinTpicVersion4.32a
{\unitlength 0.1in%
\begin{picture}(20.8100,12.1900)(18.1000,-17.2300)%
% VECTOR 2 0 3 0 Black White  
% 2 1911 685 2719 685
% 
\special{pn 8}%
\special{pa 1911 685}%
\special{pa 2719 685}%
\special{fp}%
\special{sh 1}%
\special{pa 2719 685}%
\special{pa 2652 665}%
\special{pa 2666 685}%
\special{pa 2652 705}%
\special{pa 2719 685}%
\special{fp}%
% VECTOR 2 0 3 0 Black White  
% 2 2719 1640 1845 766
% 
\special{pn 8}%
\special{pa 2719 1640}%
\special{pa 1845 766}%
\special{fp}%
\special{sh 1}%
\special{pa 1845 766}%
\special{pa 1878 827}%
\special{pa 1883 804}%
\special{pa 1906 799}%
\special{pa 1845 766}%
\special{fp}%
% VECTOR 2 0 3 0 Black White  
% 2 3831 1594 3831 786
% 
\special{pn 8}%
\special{pa 3831 1594}%
\special{pa 3831 786}%
\special{fp}%
\special{sh 1}%
\special{pa 3831 786}%
\special{pa 3811 853}%
\special{pa 3831 839}%
\special{pa 3851 853}%
\special{pa 3831 786}%
\special{fp}%
% VECTOR 2 0 3 0 Black White  
% 2 2922 1696 3730 1696
% 
\special{pn 8}%
\special{pa 2922 1696}%
\special{pa 3730 1696}%
\special{fp}%
\special{sh 1}%
\special{pa 3730 1696}%
\special{pa 3663 1676}%
\special{pa 3677 1696}%
\special{pa 3663 1716}%
\special{pa 3730 1696}%
\special{fp}%
% VECTOR 2 0 3 0 Black White  
% 2 2780 786 2780 1594
% 
\special{pn 8}%
\special{pa 2780 786}%
\special{pa 2780 1594}%
\special{fp}%
\special{sh 1}%
\special{pa 2780 1594}%
\special{pa 2800 1527}%
\special{pa 2780 1541}%
\special{pa 2760 1527}%
\special{pa 2780 1594}%
\special{fp}%
% VECTOR 2 0 3 0 Black White  
% 2 3730 685 2922 685
% 
\special{pn 8}%
\special{pa 3730 685}%
\special{pa 2922 685}%
\special{fp}%
\special{sh 1}%
\special{pa 2922 685}%
\special{pa 2989 705}%
\special{pa 2975 685}%
\special{pa 2989 665}%
\special{pa 2922 685}%
\special{fp}%
% VECTOR 2 0 3 0 Black White  
% 2 2861 786 2861 1594
% 
\special{pn 8}%
\special{pa 2861 786}%
\special{pa 2861 1594}%
\special{fp}%
\special{sh 1}%
\special{pa 2861 1594}%
\special{pa 2881 1527}%
\special{pa 2861 1541}%
\special{pa 2841 1527}%
\special{pa 2861 1594}%
\special{fp}%
% STR 2 0 3 0 Black White  
% 4 3265 584 3265 634 2 0 0 0
% $x_4$
\put(32.6500,-6.3400){\makebox(0,0)[lb]{$x_4$}}%
% STR 2 0 3 0 Black White  
% 4 3891 1140 3891 1190 2 0 0 0
% $x_3$
\put(38.9100,-11.9000){\makebox(0,0)[lb]{$x_3$}}%
% STR 2 0 3 0 Black White  
% 4 3265 1746 3265 1797 2 0 0 0
% $x_2$
\put(32.6500,-17.9700){\makebox(0,0)[lb]{$x_2$}}%
% STR 2 0 3 0 Black White  
% 4 2922 1140 2922 1190 2 0 0 0
% $x_1$
\put(29.2200,-11.9000){\makebox(0,0)[lb]{$x_1$}}%
% STR 2 0 3 0 Black White  
% 4 2659 1140 2659 1190 2 0 0 0
% $z_1$
\put(26.5900,-11.9000){\makebox(0,0)[lb]{$z_1$}}%
% STR 2 0 3 0 Black White  
% 4 2154 1241 2154 1292 2 0 0 0
% $z_2$
\put(21.5400,-12.9200){\makebox(0,0)[lb]{$z_2$}}%
% STR 2 0 3 0 Black White  
% 4 2255 584 2255 634 2 0 0 0
% $z_3$
\put(22.5500,-6.3400){\makebox(0,0)[lb]{$z_3$}}%
% STR 2 0 3 0 Black White  
% 4 2800 660 2800 710 2 0 0 0
% $1$
\put(28.0000,-7.1000){\makebox(0,0)[lb]{$1$}}%
% STR 2 0 3 0 Black White  
% 4 2800 1645 2800 1696 2 0 0 0
% $2$
\put(28.0000,-16.9600){\makebox(0,0)[lb]{$2$}}%
% STR 2 0 3 0 Black White  
% 4 1810 634 1810 685 2 0 0 0
% $5$
\put(18.1000,-6.8500){\makebox(0,0)[lb]{$5$}}%
% STR 2 0 3 0 Black White  
% 4 3795 665 3795 716 2 0 0 0
% $4$
\put(37.9500,-7.1600){\makebox(0,0)[lb]{$4$}}%
% STR 2 0 3 0 Black White  
% 4 3800 1675 3800 1726 2 0 0 0
% $3$
\put(38.0000,-17.2600){\makebox(0,0)[lb]{$3$}}%
\end{picture}}%

%% file: ex1-2.tex
%WinTpicVersion4.32a
{\unitlength 0.1in%
\begin{picture}(42.1500,24.5500)(17.4000,-29.6500)%
% VECTOR 2 0 3 0 Black White  
% 2 2153 1061 3574 1061
% 
\special{pn 8}%
\special{pa 2153 1061}%
\special{pa 3574 1061}%
\special{fp}%
\special{sh 1}%
\special{pa 3574 1061}%
\special{pa 3507 1041}%
\special{pa 3521 1061}%
\special{pa 3507 1081}%
\special{pa 3574 1061}%
\special{fp}%
% VECTOR 2 0 3 0 Black White  
% 2 5836 2690 5836 1080
% 
\special{pn 8}%
\special{pa 5836 2690}%
\special{pa 5836 1080}%
\special{fp}%
\special{sh 1}%
\special{pa 5836 1080}%
\special{pa 5816 1147}%
\special{pa 5836 1133}%
\special{pa 5856 1147}%
\special{pa 5836 1080}%
\special{fp}%
% VECTOR 2 0 3 0 Black White  
% 2 4040 2830 5651 2830
% 
\special{pn 8}%
\special{pa 4040 2830}%
\special{pa 5651 2830}%
\special{fp}%
\special{sh 1}%
\special{pa 5651 2830}%
\special{pa 5584 2810}%
\special{pa 5598 2830}%
\special{pa 5584 2850}%
\special{pa 5651 2830}%
\special{fp}%
% VECTOR 2 0 3 0 Black White  
% 2 5634 878 4023 878
% 
\special{pn 8}%
\special{pa 5634 878}%
\special{pa 4023 878}%
\special{fp}%
\special{sh 1}%
\special{pa 4023 878}%
\special{pa 4090 898}%
\special{pa 4076 878}%
\special{pa 4090 858}%
\special{pa 4023 878}%
\special{fp}%
% STR 2 0 3 0 Black White  
% 4 4706 676 4706 776 2 0 0 0
% $x_4$
\put(47.0600,-7.7600){\makebox(0,0)[lb]{$x_4$}}%
% STR 2 0 3 0 Black White  
% 4 5955 1785 5955 1884 2 0 0 0
% $x_3$
\put(59.5500,-18.8400){\makebox(0,0)[lb]{$x_3$}}%
% STR 2 0 3 0 Black White  
% 4 4706 2994 4706 3095 2 0 0 0
% $x_2$
\put(47.0600,-30.9500){\makebox(0,0)[lb]{$x_2$}}%
% STR 2 0 3 0 Black White  
% 4 4022 1865 4022 1964 2 0 0 0
% $x_1$
\put(40.2200,-19.6400){\makebox(0,0)[lb]{$x_1$}}%
% STR 2 0 3 0 Black White  
% 4 3835 1865 3835 1964 2 0 0 0
% $z_1$
\put(38.3500,-19.6400){\makebox(0,0)[lb]{$z_1$}}%
% STR 2 0 3 0 Black White  
% 4 2461 1989 2390 2060 2 0 0 0
% $z_2$
\put(23.9000,-20.6000){\rotatebox{-45.0000}{\makebox(0,0)[lb]{$z_2$}}}%
% STR 2 0 3 0 Black White  
% 4 2750 540 2750 640 2 0 0 0
% $z_3$
\put(27.5000,-6.4000){\makebox(0,0)[lb]{$z_3$}}%
% STR 2 0 3 0 Black White  
% 4 3780 828 3780 927 2 0 0 0
% $1$
\put(37.8000,-9.2700){\makebox(0,0)[lb]{$1$}}%
% STR 2 0 3 0 Black White  
% 4 3780 2792 3780 2894 2 0 0 0
% $2$
\put(37.8000,-28.9400){\makebox(0,0)[lb]{$2$}}%
% STR 2 0 3 0 Black White  
% 4 1740 888 1740 990 2 0 0 0
% $5$
\put(17.4000,-9.9000){\makebox(0,0)[lb]{$5$}}%
% STR 2 0 3 0 Black White  
% 4 5763 838 5763 939 2 0 0 0
% $4$
\put(57.6300,-9.3900){\makebox(0,0)[lb]{$4$}}%
% STR 2 0 3 0 Black White  
% 4 5773 2853 5773 2954 2 0 0 0
% $3$
\put(57.7300,-29.5400){\makebox(0,0)[lb]{$3$}}%
% VECTOR 2 0 3 0 Black White  
% 2 3440 2505 2196 1262
% 
\special{pn 8}%
\special{pa 3440 2505}%
\special{pa 2196 1262}%
\special{fp}%
\special{sh 1}%
\special{pa 2196 1262}%
\special{pa 2229 1323}%
\special{pa 2234 1300}%
\special{pa 2257 1295}%
\special{pa 2196 1262}%
\special{fp}%
% VECTOR 2 0 3 0 Black White  
% 2 3271 2673 2028 1430
% 
\special{pn 8}%
\special{pa 3271 2673}%
\special{pa 2028 1430}%
\special{fp}%
\special{sh 1}%
\special{pa 2028 1430}%
\special{pa 2061 1491}%
\special{pa 2066 1468}%
\special{pa 2089 1463}%
\special{pa 2028 1430}%
\special{fp}%
% VECTOR 2 0 3 0 Black White  
% 2 3103 2842 1859 1599
% 
\special{pn 8}%
\special{pa 3103 2842}%
\special{pa 1859 1599}%
\special{fp}%
\special{sh 1}%
\special{pa 1859 1599}%
\special{pa 1892 1660}%
\special{pa 1897 1637}%
\special{pa 1920 1632}%
\special{pa 1859 1599}%
\special{fp}%
% VECTOR 2 0 3 0 Black White  
% 2 2153 878 3574 878
% 
\special{pn 8}%
\special{pa 2153 878}%
\special{pa 3574 878}%
\special{fp}%
\special{sh 1}%
\special{pa 3574 878}%
\special{pa 3507 858}%
\special{pa 3521 878}%
\special{pa 3507 898}%
\special{pa 3574 878}%
\special{fp}%
% VECTOR 2 0 3 0 Black White  
% 2 2153 695 3574 695
% 
\special{pn 8}%
\special{pa 2153 695}%
\special{pa 3574 695}%
\special{fp}%
\special{sh 1}%
\special{pa 3574 695}%
\special{pa 3507 675}%
\special{pa 3521 695}%
\special{pa 3507 715}%
\special{pa 3574 695}%
\special{fp}%
% VECTOR 2 0 3 0 Black White  
% 2 3636 1141 3636 2562
% 
\special{pn 8}%
\special{pa 3636 1141}%
\special{pa 3636 2562}%
\special{fp}%
\special{sh 1}%
\special{pa 3636 2562}%
\special{pa 3656 2495}%
\special{pa 3636 2509}%
\special{pa 3616 2495}%
\special{pa 3636 2562}%
\special{fp}%
% VECTOR 2 0 3 0 Black White  
% 2 3819 1141 3819 2562
% 
\special{pn 8}%
\special{pa 3819 1141}%
\special{pa 3819 2562}%
\special{fp}%
\special{sh 1}%
\special{pa 3819 2562}%
\special{pa 3839 2495}%
\special{pa 3819 2509}%
\special{pa 3799 2495}%
\special{pa 3819 2562}%
\special{fp}%
% VECTOR 2 0 3 0 Black White  
% 2 4003 1141 4003 2562
% 
\special{pn 8}%
\special{pa 4003 1141}%
\special{pa 4003 2562}%
\special{fp}%
\special{sh 1}%
\special{pa 4003 2562}%
\special{pa 4023 2495}%
\special{pa 4003 2509}%
\special{pa 3983 2495}%
\special{pa 4003 2562}%
\special{fp}%
% VECTOR 2 0 3 0 Black White  
% 2 4188 2218 5132 1118
% 
\special{pn 8}%
\special{pa 4188 2218}%
\special{pa 5132 1118}%
\special{fp}%
\special{sh 1}%
\special{pa 5132 1118}%
\special{pa 5073 1156}%
\special{pa 5097 1158}%
\special{pa 5104 1182}%
\special{pa 5132 1118}%
\special{fp}%
% VECTOR 2 0 3 0 Black White  
% 2 4372 2402 5315 1301
% 
\special{pn 8}%
\special{pa 4372 2402}%
\special{pa 5315 1301}%
\special{fp}%
\special{sh 1}%
\special{pa 5315 1301}%
\special{pa 5256 1339}%
\special{pa 5280 1342}%
\special{pa 5287 1365}%
\special{pa 5315 1301}%
\special{fp}%
% VECTOR 2 0 3 0 Black White  
% 2 5472 1485 4529 2585
% 
\special{pn 8}%
\special{pa 5472 1485}%
\special{pa 4529 2585}%
\special{fp}%
\special{sh 1}%
\special{pa 4529 2585}%
\special{pa 4588 2547}%
\special{pa 4564 2545}%
\special{pa 4557 2521}%
\special{pa 4529 2585}%
\special{fp}%
% VECTOR 2 0 3 0 Black White  
% 2 5235 2562 4249 1501
% 
\special{pn 8}%
\special{pa 5235 2562}%
\special{pa 4249 1501}%
\special{fp}%
\special{sh 1}%
\special{pa 4249 1501}%
\special{pa 4280 1563}%
\special{pa 4285 1540}%
\special{pa 4309 1536}%
\special{pa 4249 1501}%
\special{fp}%
% VECTOR 2 0 3 0 Black White  
% 2 5439 2401 4452 1339
% 
\special{pn 8}%
\special{pa 5439 2401}%
\special{pa 4452 1339}%
\special{fp}%
\special{sh 1}%
\special{pa 4452 1339}%
\special{pa 4483 1401}%
\special{pa 4488 1378}%
\special{pa 4512 1374}%
\special{pa 4452 1339}%
\special{fp}%
% VECTOR 2 0 3 0 Black White  
% 2 4651 1204 5639 2265
% 
\special{pn 8}%
\special{pa 4651 1204}%
\special{pa 5639 2265}%
\special{fp}%
\special{sh 1}%
\special{pa 5639 2265}%
\special{pa 5608 2203}%
\special{pa 5603 2226}%
\special{pa 5579 2230}%
\special{pa 5639 2265}%
\special{fp}%
% STR 2 0 3 0 Black White  
% 4 3350 1847 3350 1947 2 0 0 0
% $y_{z_2z_3}$
\put(33.5000,-19.4700){\makebox(0,0)[lb]{$y_{z_2z_3}$}}%
% STR 2 0 3 0 Black White  
% 4 2740 730 2740 830 2 0 0 0
% $y_{z_1z_2}$
\put(27.4000,-8.3000){\makebox(0,0)[lb]{$y_{z_1z_2}$}}%
% STR 2 0 3 0 Black White  
% 4 2740 930 2740 1030 2 0 0 0
% $y_{x_1z_2}$
\put(27.4000,-10.3000){\makebox(0,0)[lb]{$y_{x_1z_2}$}}%
% STR 2 0 3 0 Black White  
% 4 2601 1819 2530 1890 2 0 0 0
% $y_{z_3z_1}$
\put(25.3000,-18.9000){\rotatebox{-45.0000}{\makebox(0,0)[lb]{$y_{z_3z_1}$}}}%
% STR 2 0 3 0 Black White  
% 4 2771 1639 2700 1710 2 0 0 0
% $y_{z_3x_1}$
\put(27.0000,-17.1000){\rotatebox{-45.0000}{\makebox(0,0)[lb]{$y_{z_3x_1}$}}}%
% STR 2 0 3 0 Black White  
% 4 4871 1179 4950 1240 2 0 0 0
% $y_{x_4x_1}$
\put(49.5000,-12.4000){\rotatebox{52.3264}{\makebox(0,0)[lb]{$y_{x_4x_1}$}}}%
% STR 2 0 3 0 Black White  
% 4 5071 1379 5150 1440 2 0 0 0
% $y_{x_4z_1}$
\put(51.5000,-14.4000){\rotatebox{52.3264}{\makebox(0,0)[lb]{$y_{x_4z_1}$}}}%
% STR 2 0 3 0 Black White  
% 4 4484 2425 4560 2490 2 0 0 0
% $y_{x_2x_3}$
\put(45.6000,-24.9000){\rotatebox{49.4609}{\makebox(0,0)[lb]{$y_{x_2x_3}$}}}%
% STR 2 0 3 0 Black White  
% 4 4569 1300 4498 1370 2 0 0 0
% $y_{z_1x_2}$
\put(44.9800,-13.7000){\rotatebox{-45.4064}{\makebox(0,0)[lb]{$y_{z_1x_2}$}}}%
% STR 2 0 3 0 Black White  
% 4 4387 1490 4315 1560 2 0 0 0
% $y_{x_1x_2}$
\put(43.1500,-15.6000){\rotatebox{-45.8069}{\makebox(0,0)[lb]{$y_{x_1x_2}$}}}%
% STR 2 0 3 0 Black White  
% 4 5461 1900 5392 1971 2 0 0 0
% $y_{x_3x_4}$
\put(53.9200,-19.7100){\rotatebox{-44.1815}{\makebox(0,0)[lb]{$y_{x_3x_4}$}}}%
\end{picture}}%

%% file: ex2-1.tex
%WinTpicVersion4.32a
{\unitlength 0.1in%
\begin{picture}(22.8100,13.9300)(16.1000,-17.2300)%
% VECTOR 2 0 3 0 Black White  
% 2 1911 585 2719 585
% 
\special{pn 8}%
\special{pa 1911 585}%
\special{pa 2719 585}%
\special{fp}%
\special{sh 1}%
\special{pa 2719 585}%
\special{pa 2652 565}%
\special{pa 2666 585}%
\special{pa 2652 605}%
\special{pa 2719 585}%
\special{fp}%
% VECTOR 2 0 3 0 Black White  
% 2 3831 1594 3831 786
% 
\special{pn 8}%
\special{pa 3831 1594}%
\special{pa 3831 786}%
\special{fp}%
\special{sh 1}%
\special{pa 3831 786}%
\special{pa 3811 853}%
\special{pa 3831 839}%
\special{pa 3851 853}%
\special{pa 3831 786}%
\special{fp}%
% VECTOR 2 0 3 0 Black White  
% 2 2922 1696 3730 1696
% 
\special{pn 8}%
\special{pa 2922 1696}%
\special{pa 3730 1696}%
\special{fp}%
\special{sh 1}%
\special{pa 3730 1696}%
\special{pa 3663 1676}%
\special{pa 3677 1696}%
\special{pa 3663 1716}%
\special{pa 3730 1696}%
\special{fp}%
% VECTOR 2 0 3 0 Black White  
% 2 3730 685 2922 685
% 
\special{pn 8}%
\special{pa 3730 685}%
\special{pa 2922 685}%
\special{fp}%
\special{sh 1}%
\special{pa 2922 685}%
\special{pa 2989 705}%
\special{pa 2975 685}%
\special{pa 2989 665}%
\special{pa 2922 685}%
\special{fp}%
% VECTOR 2 0 3 0 Black White  
% 2 2861 786 2861 1594
% 
\special{pn 8}%
\special{pa 2861 786}%
\special{pa 2861 1594}%
\special{fp}%
\special{sh 1}%
\special{pa 2861 1594}%
\special{pa 2881 1527}%
\special{pa 2861 1541}%
\special{pa 2841 1527}%
\special{pa 2861 1594}%
\special{fp}%
% STR 2 0 3 0 Black White  
% 4 3265 584 3265 634 2 0 0 0
% $x_4$
\put(32.6500,-6.3400){\makebox(0,0)[lb]{$x_4$}}%
% STR 2 0 3 0 Black White  
% 4 3891 1140 3891 1190 2 0 0 0
% $x_3$
\put(38.9100,-11.9000){\makebox(0,0)[lb]{$x_3$}}%
% STR 2 0 3 0 Black White  
% 4 3265 1746 3265 1797 2 0 0 0
% $x_2$
\put(32.6500,-17.9700){\makebox(0,0)[lb]{$x_2$}}%
% STR 2 0 3 0 Black White  
% 4 2922 1140 2922 1190 2 0 0 0
% $x_1$
\put(29.2200,-11.9000){\makebox(0,0)[lb]{$x_1$}}%
% STR 2 0 3 0 Black White  
% 4 2185 494 2185 544 2 0 0 0
% $z_1$
\put(21.8500,-5.4400){\makebox(0,0)[lb]{$z_1$}}%
% STR 2 0 3 0 Black White  
% 4 2800 660 2800 710 2 0 0 0
% $1$
\put(28.0000,-7.1000){\makebox(0,0)[lb]{$1$}}%
% STR 2 0 3 0 Black White  
% 4 2800 1645 2800 1696 2 0 0 0
% $2$
\put(28.0000,-16.9600){\makebox(0,0)[lb]{$2$}}%
% STR 2 0 3 0 Black White  
% 4 1610 699 1610 750 2 0 0 0
% $5$
\put(16.1000,-7.5000){\makebox(0,0)[lb]{$5$}}%
% STR 2 0 3 0 Black White  
% 4 3795 665 3795 716 2 0 0 0
% $4$
\put(37.9500,-7.1600){\makebox(0,0)[lb]{$4$}}%
% STR 2 0 3 0 Black White  
% 4 3800 1675 3800 1726 2 0 0 0
% $3$
\put(38.0000,-17.2600){\makebox(0,0)[lb]{$3$}}%
% VECTOR 2 0 3 0 Black White  
% 2 1910 780 2718 780
% 
\special{pn 8}%
\special{pa 1910 780}%
\special{pa 2718 780}%
\special{fp}%
\special{sh 1}%
\special{pa 2718 780}%
\special{pa 2651 760}%
\special{pa 2665 780}%
\special{pa 2651 800}%
\special{pa 2718 780}%
\special{fp}%
% STR 2 0 3 0 Black White  
% 4 2170 900 2170 950 2 0 0 0
% $z_m$
\put(21.7000,-9.5000){\makebox(0,0)[lb]{$z_m$}}%
% STR 2 0 3 0 Black White  
% 4 2240 770 2340 770 2 0 0 0
% $\cdots$
\put(23.4000,-7.7000){\rotatebox{90.0000}{\makebox(0,0)[lb]{$\cdots$}}}%
\end{picture}}%

%% file: ex2-2.tex
%WinTpicVersion4.32a
{\unitlength 0.1in%
\begin{picture}(42.1500,24.8500)(17.4000,-29.6500)%
% VECTOR 2 0 3 0 Black White  
% 2 2153 981 3574 981
% 
\special{pn 8}%
\special{pa 2153 981}%
\special{pa 3574 981}%
\special{fp}%
\special{sh 1}%
\special{pa 3574 981}%
\special{pa 3507 961}%
\special{pa 3521 981}%
\special{pa 3507 1001}%
\special{pa 3574 981}%
\special{fp}%
% VECTOR 2 0 3 0 Black White  
% 2 5836 2690 5836 1080
% 
\special{pn 8}%
\special{pa 5836 2690}%
\special{pa 5836 1080}%
\special{fp}%
\special{sh 1}%
\special{pa 5836 1080}%
\special{pa 5816 1147}%
\special{pa 5836 1133}%
\special{pa 5856 1147}%
\special{pa 5836 1080}%
\special{fp}%
% VECTOR 2 0 3 0 Black White  
% 2 4040 2830 5651 2830
% 
\special{pn 8}%
\special{pa 4040 2830}%
\special{pa 5651 2830}%
\special{fp}%
\special{sh 1}%
\special{pa 5651 2830}%
\special{pa 5584 2810}%
\special{pa 5598 2830}%
\special{pa 5584 2850}%
\special{pa 5651 2830}%
\special{fp}%
% VECTOR 2 0 3 0 Black White  
% 2 5634 878 4023 878
% 
\special{pn 8}%
\special{pa 5634 878}%
\special{pa 4023 878}%
\special{fp}%
\special{sh 1}%
\special{pa 4023 878}%
\special{pa 4090 898}%
\special{pa 4076 878}%
\special{pa 4090 858}%
\special{pa 4023 878}%
\special{fp}%
% STR 2 0 3 0 Black White  
% 4 4706 676 4706 776 2 0 0 0
% $x_4$
\put(47.0600,-7.7600){\makebox(0,0)[lb]{$x_4$}}%
% STR 2 0 3 0 Black White  
% 4 5955 1785 5955 1884 2 0 0 0
% $x_3$
\put(59.5500,-18.8400){\makebox(0,0)[lb]{$x_3$}}%
% STR 2 0 3 0 Black White  
% 4 4706 2994 4706 3095 2 0 0 0
% $x_2$
\put(47.0600,-30.9500){\makebox(0,0)[lb]{$x_2$}}%
% STR 2 0 3 0 Black White  
% 4 3820 1861 3820 1960 2 0 0 0
% $x_1$
\put(38.2000,-19.6000){\makebox(0,0)[lb]{$x_1$}}%
% STR 2 0 3 0 Black White  
% 4 2750 510 2750 610 2 0 0 0
% $z_1$
\put(27.5000,-6.1000){\makebox(0,0)[lb]{$z_1$}}%
% STR 2 0 3 0 Black White  
% 4 3780 828 3780 927 2 0 0 0
% $1$
\put(37.8000,-9.2700){\makebox(0,0)[lb]{$1$}}%
% STR 2 0 3 0 Black White  
% 4 3780 2792 3780 2894 2 0 0 0
% $2$
\put(37.8000,-28.9400){\makebox(0,0)[lb]{$2$}}%
% STR 2 0 3 0 Black White  
% 4 1740 888 1740 990 2 0 0 0
% $5$
\put(17.4000,-9.9000){\makebox(0,0)[lb]{$5$}}%
% STR 2 0 3 0 Black White  
% 4 5763 838 5763 939 2 0 0 0
% $4$
\put(57.6300,-9.3900){\makebox(0,0)[lb]{$4$}}%
% STR 2 0 3 0 Black White  
% 4 5773 2853 5773 2954 2 0 0 0
% $3$
\put(57.7300,-29.5400){\makebox(0,0)[lb]{$3$}}%
% VECTOR 2 0 3 0 Black White  
% 2 3530 2445 2286 1202
% 
\special{pn 8}%
\special{pa 3530 2445}%
\special{pa 2286 1202}%
\special{fp}%
\special{sh 1}%
\special{pa 2286 1202}%
\special{pa 2319 1263}%
\special{pa 2324 1240}%
\special{pa 2347 1235}%
\special{pa 2286 1202}%
\special{fp}%
% VECTOR 2 0 3 0 Black White  
% 2 3271 2673 2028 1430
% 
\special{pn 8}%
\special{pa 3271 2673}%
\special{pa 2028 1430}%
\special{fp}%
\special{sh 1}%
\special{pa 2028 1430}%
\special{pa 2061 1491}%
\special{pa 2066 1468}%
\special{pa 2089 1463}%
\special{pa 2028 1430}%
\special{fp}%
% VECTOR 2 0 3 0 Black White  
% 2 2150 660 3571 660
% 
\special{pn 8}%
\special{pa 2150 660}%
\special{pa 3571 660}%
\special{fp}%
\special{sh 1}%
\special{pa 3571 660}%
\special{pa 3504 640}%
\special{pa 3518 660}%
\special{pa 3504 680}%
\special{pa 3571 660}%
\special{fp}%
% VECTOR 2 0 3 0 Black White  
% 2 3819 1141 3819 2562
% 
\special{pn 8}%
\special{pa 3819 1141}%
\special{pa 3819 2562}%
\special{fp}%
\special{sh 1}%
\special{pa 3819 2562}%
\special{pa 3839 2495}%
\special{pa 3819 2509}%
\special{pa 3799 2495}%
\special{pa 3819 2562}%
\special{fp}%
% VECTOR 2 0 3 0 Black White  
% 2 4188 2218 5132 1118
% 
\special{pn 8}%
\special{pa 4188 2218}%
\special{pa 5132 1118}%
\special{fp}%
\special{sh 1}%
\special{pa 5132 1118}%
\special{pa 5073 1156}%
\special{pa 5097 1158}%
\special{pa 5104 1182}%
\special{pa 5132 1118}%
\special{fp}%
% VECTOR 2 0 3 0 Black White  
% 2 5472 1485 4529 2585
% 
\special{pn 8}%
\special{pa 5472 1485}%
\special{pa 4529 2585}%
\special{fp}%
\special{sh 1}%
\special{pa 4529 2585}%
\special{pa 4588 2547}%
\special{pa 4564 2545}%
\special{pa 4557 2521}%
\special{pa 4529 2585}%
\special{fp}%
% VECTOR 2 0 3 0 Black White  
% 2 5235 2562 4249 1501
% 
\special{pn 8}%
\special{pa 5235 2562}%
\special{pa 4249 1501}%
\special{fp}%
\special{sh 1}%
\special{pa 4249 1501}%
\special{pa 4280 1563}%
\special{pa 4285 1540}%
\special{pa 4309 1536}%
\special{pa 4249 1501}%
\special{fp}%
% VECTOR 2 0 3 0 Black White  
% 2 4651 1204 5639 2265
% 
\special{pn 8}%
\special{pa 4651 1204}%
\special{pa 5639 2265}%
\special{fp}%
\special{sh 1}%
\special{pa 5639 2265}%
\special{pa 5608 2203}%
\special{pa 5603 2226}%
\special{pa 5579 2230}%
\special{pa 5639 2265}%
\special{fp}%
% STR 2 0 3 0 Black White  
% 4 2551 2109 2480 2180 2 0 0 0
% $y_{z_mx_1}$
\put(24.8000,-21.8000){\rotatebox{-45.0000}{\makebox(0,0)[lb]{$y_{z_mx_1}$}}}%
% STR 2 0 3 0 Black White  
% 4 3001 1699 2930 1770 2 0 0 0
% $y_{z_1x_1}$
\put(29.3000,-17.7000){\rotatebox{-45.0000}{\makebox(0,0)[lb]{$y_{z_1x_1}$}}}%
% STR 2 0 3 0 Black White  
% 4 4871 1179 4950 1240 2 0 0 0
% $y_{x_4x_1}$
\put(49.5000,-12.4000){\rotatebox{52.3264}{\makebox(0,0)[lb]{$y_{x_4x_1}$}}}%
% STR 2 0 3 0 Black White  
% 4 4484 2425 4560 2490 2 0 0 0
% $y_{x_2x_3}$
\put(45.6000,-24.9000){\rotatebox{49.4609}{\makebox(0,0)[lb]{$y_{x_2x_3}$}}}%
% STR 2 0 3 0 Black White  
% 4 4387 1490 4315 1560 2 0 0 0
% $y_{x_1x_2}$
\put(43.1500,-15.6000){\rotatebox{-45.8069}{\makebox(0,0)[lb]{$y_{x_1x_2}$}}}%
% STR 2 0 3 0 Black White  
% 4 5461 1900 5392 1971 2 0 0 0
% $y_{x_3x_4}$
\put(53.9200,-19.7100){\rotatebox{-44.1815}{\makebox(0,0)[lb]{$y_{x_3x_4}$}}}%
% STR 2 0 3 0 Black White  
% 4 2750 1040 2750 1140 2 0 0 0
% $z_m$
\put(27.5000,-11.4000){\makebox(0,0)[lb]{$z_m$}}%
% STR 2 0 3 0 Black White  
% 4 2850 930 2950 930 2 0 0 0
% $\cdots$
\put(29.5000,-9.3000){\rotatebox{90.0000}{\makebox(0,0)[lb]{$\cdots$}}}%
% STR 2 0 3 0 Black White  
% 4 2671 1971 2742 2042 2 0 0 0
% $\cdots$
\put(27.4200,-20.4200){\rotatebox{45.0000}{\makebox(0,0)[lb]{$\cdots$}}}%
\end{picture}}%